\documentclass[11pt,leqno]{amsart}
\usepackage{amsmath}

\usepackage{amsfonts}
\usepackage{mathtools}
\usepackage{amssymb}
\usepackage{mathrsfs}
\usepackage{subcaption}  
\usepackage{bookmark} 

\usepackage[a4paper, centering]{geometry}
\usepackage{epstopdf}
\usepackage{amscd}
\usepackage[dvipsnames]{xcolor}
\usepackage{graphicx}
\usepackage{yfonts}

\newtheorem{teo}{Theorem}[section]
\newtheorem{lem}[teo]{Lemma}
\newtheorem{cor}[teo]{Corollary}
\newtheorem{prop}[teo]{Proposition}

\newcommand{\medint}{-\kern  -,395cm\int}

\theoremstyle{remark}
\newtheorem{oss}[teo]{Remark}

\theoremstyle{definition}
\newtheorem{defi}[teo]{Definition}

\theoremstyle{definition}

\newcommand{\tildew}{\tilde w}
\newcommand{\wstar}{{\hat w}}

\newcommand{\supp}{\mathrm{supp}}

\newcommand{\bbR}{{\mathbb{R}}}
\newcommand{\bbN}{{\mathbb{N}}}

\usepackage{mathtools}
\usepackage{esint}

\newcommand{\R}{\bbR}

\newcommand{\ov}[1]{\overline{#1}}

\def\de{{\rm d}}

\newcommand{\norma}[1]{{\left\| #1\right\|}}

\newcommand{\hol}[1]{\"{#1}}

\makeatletter
\def\cleardoublepage{\clearpage\if@twoside \ifodd\c@page\else
\hbox{}
\thispagestyle{empty}
\newpage
\if@twocolumn\hbox{}\newpage\fi\fi\fi}
\makeatother
\title{Minimization of Degenerate Nonlinear Functionals under Radial Symmetry}

\author[V.~Chiad\`o Piat]{Valeria Chiad\`o Piat}
\address{Dipartimento di scienze Matematiche Giuseppe Lagrange, Corso duca degli abruzzi, 24, Torino (Italy)}
\email{valeria.chiadopiat@polito.it}\author[V.~De Cicco]{Virginia De Cicco}
\address{Dipartimento di Scienze di Base  e Applicate per l'Ingegneria, Sapienza Univ.\ di Roma\\
	Via A.\ Scarpa 10 -- I-00185 Roma (Italy)}
\email{virginia.decicco@uniroma1.it}
\author[A.~Melchor Hernandez]{Anderson Melchor Hernandez}
\address{Dipartimento di Matematica, Via Zamboni, 33, 40126, Bologna (Italy)}
\email{anderson.melchor@unibo.it}


\begin{document}
\begin{abstract}
In this work, we study the minimization of nonlinear functionals in dimension $d\geq 1$ that depend on a degenerate radial weight $w$. Our goal is to prove the existence of minimizers in a suitable functional class here introduced and to establish that the minimizers of such functionals, which exhibit $p$-growth with $1 < p < +\infty$, are radially symmetric. In our analysis, we adopt the approach developed in \cite{CCMH1,CC}, where $w$  does not satisfy classical assumptions such as doubling or Muckenhoupt conditions.  The core of our method relies on proving the validity of a weighted Poincaré inequality involving a suitably constructed auxiliary weight.
\end{abstract}
\maketitle
\tableofcontents
\section{Introduction}
In this work, we consider the analysis of nonlinear functionals in dimension $d\geq 1$, allowing for a degenerate weight $w$.  Our main result concerns the minimization of a functional (see formula \eqref{scopus} below) involving the lower semicontinuous envelope of $F$, denoted as $\overline{F}$,  where $F$ satisfies $p$-growth for $1<p<+\infty$. More precisely, we consider
\begin{align*}
F(u)\coloneqq\left\{
\begin{aligned}
& \int_{\Omega} |\nabla u|^{p}w \de x\text{ \ \ if } u\in AC_{r}^{d}(\Omega),\\
& +\infty \ \ \ \ \ \ \ \ \ \ \ \ \ \ \text{if}\hskip 0,1cm  u\in X\setminus AC_{r}^{d}(\Omega),
\end{aligned}
\right.
\end{align*}
where $\Omega$ is an open bounded set in $\mathbb{R}^{d}$ and rotational invariant,  $w$ is a nonnegative, locally integrable radial function, and $X$ is a topological space comprising measurable functions which will be introduced later on, and $AC_{r}^{d}(\Omega)$ is the space of radial functions (introduced in \eqref{spazior} in Sect. 2) contained in the class $AC^{d}(\Omega)$ of the $d$-absolutely continuous functions in the sense of Mal{\'y}, see \cite[Page 2]{Mal}. 

Before giving all the precise details of our main result, let us give some review about the study of functionals like $F$ and $\overline{F}$.
Several works have investigated the functional above by adopting different functional frameworks; see, for instance, \cite{CD,CCMH1,CCMH2,CC,FM,Ha,Ma}. Particular attention has been devoted to the case $p=2$, which is considered canonical due to its connection with probabilistic issues involving the so-called Dirichlet forms \cite{ALB,RO}. Further recent works have been dedicated in the analysis of a weaker form of $F$ where the gradient is replaced by the upper gradient \cite{LAMB2,LAMB3,LAMB1}. These works have been given a new field of research where the authors have aimed to extending the theory of regularization by heat kernel to metric measure spaces. In doing that, some important difficulties taking place, for instance, what is the natural definition of Sobolev space when $\Omega$ is a metric space of arbitrary dimension supporting a generic measure $\mu$, say,  $\mu(\de x)=w(x)\de x$, see \cite{HKST}. The theory of weighted Sobolev spaces in infinite-dimensional spaces is known for instance when the measure $\mu$ is the gaussian measure which allows even to define space of bounded variation, see for instance \cite{Caselles}. 
Recently some works have been dedicated to establish a theory of weighted Sobolev spaces where the reference measure  do not satisfy any doubling or Poincaré inequality \cite{LAMB0}. As a matter of fact, they proposed the definition of weighted Sobolev spaces, defined over a complete and separable metric space equipped with a boundedly-finite Borel measure using  three different approaches: via approximation with Lipschitz functions; by studying the behaviour along curves; via integration-by-parts, using Lipschitz derivations with divergence. Furthermore, they proved (see \cite[Theorem 7.1]{LAMB0}) the equivalence of all definitions. Let us mention that the case of approximation with Lipschitz functions is related to the relaxation of $F$.
In general, the identification of the functional ${\overline F}$ is a challenging task, and some authors have been used the density of $C^{1}$-functions in weighted Sobolev spaces as an important tool, see for instance, \cite{CPSC,CC}.  In  such an approach, however, some additional assumptions on $w$, as  described in \cite{CPSC}, were necessary. For example, to prove the density of $C^{1}$-functions, it is sometimes assumed that $w$ satisfies the doubling or Muckenhoupt conditions \cite{HKST,Muk}. Alternatively, in \cite{CC}, have been adopted the case where such requirements on $w$ are not satisfied, $p=2$, and where $X$ is not fixed a priori.

Recently, in \cite{CCMH1} an explicit formula for $\ov{F}$  have been obtained in the onedimensional case with $1<p<+\infty$. We emphasize that, a priori, the choice of $X$ strongly depends on the weight $w$. In fact, as observed in \cite{CCMH1}, $X$ can be defined with respect to an additional weight, denoted by $\wstar_{p}$. As noted by the authors, this function plays a crucial role in compensating for the degeneracy of $w$ and allows for a proper characterization of the domain of $\overline{F}$.
At this stage, let us highlight a key difference between our approach to relaxing $F$ and the one developed in \cite{LAMB0}. In our relaxation procedure, we establish a Poincaré inequality that allows us to identify the appropriate topology in which to relax $F$. Moreover, our density argument involves two distinct measures, in contrast to \cite[Definition 5.2]{LAMB0}. Let us also notice that the approach via density proposed in \cite[Subsection 5.1]{LAMB0} comes from \cite{LAMB4}. However, let us mention that in \cite{LAMB4} the authors were involved with metric spaces where a doubling condition is assumed. On the other hand, we also observe that, since we work in a finite-dimensional setting, there is no need to define Sobolev spaces via relaxation, as is done in that work. 

In the present work, we then aim to analyze the minimizer of 
\begin{equation}\label{scopus}
H(u):=\overline{F}(u)+\|u-g\|_{L^{p}(\Omega,(\wstar_{p})^{p-1})},
\end{equation}
\vskip -0,1cm
where $g\in L^{p}(\Omega,(\wstar_{p})^{p-1})$ is a given radial function. To this end, we address the multidimensional extension of the results from \cite{CCMH1} about the relaxation of $F$. In the one-dimensional setting considered in that work, the embedding of Sobolev functions into absolutely continuous functions played a crucial role in deriving an explicit formula for $\overline{F}$. However, it is worth noting that this property does not hold in higher dimensions, making the extension of those results to the multivariate case far from straightforward. In particular, let us recall that the notion of absolute continuity in higher dimensions differs significantly from the one-dimensional case \cite{Mal}. \\
We now outline our approach to study the minimizer of  $H$. In our setting, we restrict ourselves to the case of a radial weight $w$, which can be written as $w(\cdot) = \eta(|\cdot|)$, where $\eta$ is a measurable function of a single variable. Furthermore, we construct an additional radial weight, denoted by $\wstar_{p}$, which will be used to define the space $X$, and the forcing term in \eqref{scopus} involving a generic function $g\in L^{p}(\Omega,(\wstar_{p})^{p-1})$. This particular structure of $w$ enables us to introduce the space $AC_{r}^{d}$, consisting of radial functions that are absolutely continuous on $\supp(\eta)$. This space of functions is new, and it is a subset of the space of $d$-absolutely continuous functions of $d$-bounded variation defined in \cite{Mal}, see Lemma \autoref{bvdspace} below. By following \cite{CCMH1}, we then need to prove Poincar\'e inequalities involving $w$ and $\hat{w}_{p}$. Specifically, we consider the $p$-norm of the gradient term of a generic function $u$ weighted by $w$, while the $p$-norm of $u$ itself is weighted by $(\hat{w}_{p})^{p-1}$. Subsequently, assuming that $w$ is finitely degenerate (see Definition \ref{deff1} below), we proceed to choose $X=L^{p}((\hat{w}_{p})^{p-1})$, and we show that $AC_{r}^{d}$-functions are dense, in a suitable Sobolev space $W\subseteq X$ (see formula \eqref{eq:spaceW} below). Therefore, we obtain an explicit identification of the domain of $\overline{F}$ performing the relaxation in the strong topology of $X$. The resulting relaxed functional $\overline{F}$ still maintains the same form of $F$ as in the onedimensional case considered in \cite{CCMH1}, and its ambient space consists of functions that are of $L^{p}((\hat{w}_{p})^{p-1})$-integrable type. 

Lastly, after providing the explicit expression for $\overline{F}$ we then aim to prove that the minimizer of $H$ is a radial function among all possible functions in $W_{{\rm loc}}^{1,1}(\Omega)$ whose gradient belongs to $L_{{\rm loc}}^{p}(\Omega)$, see Theorem \autoref{thm:minimizer}.  The fact that $\overline{F}$ is given explicitly allows us to proceed in proving our minimization result by employing the lines of \cite{CrastaMalusa}.

This work is structured as follows. In Section \ref{sec:3}, we study the validity of Poincar\'e inequalities with double weight, see Theorem \ref{Poin1} below. In Section \ref{relaxp}, we formulate and prove a relaxation theorem, see Theorem \ref{mainfinitelocsum} below. In Section \ref{sec:minimizer}, we prove firstly that there exists a radial minimizer in the class of radial competitors and then in the class of general functions, and eventually that the minimizer of $H$ is radial.

\section{Setting and preliminaries}\label{sec:2}
\subsection{A radial weight}
Let $d\geq 1$ be a natural number, and consider $\Omega$ to be an open bounded subset of $\bbR^{d}$ which is invariant by rotation, i.e.
$$
\Omega=\{x\in \R^d: |x|< b
\}\,\ \ \text{ or }\ \ \ 
\Omega=\{x\in \R^d: 0\leq a< |x|< b
\},
$$
with $a,b\in \R$.
Given $1<p<\infty$, we let $\frac{1}{p'}= 1-\frac{1}{p}$. In what follows, we consider a radial weight $w:\R^{d} \to \R$ satisfying
\begin{equation}\label{minassweight}
w(x)\coloneqq \eta(\vert x\vert), \hskip 0,1cm \eta\geq 0 \text{ a.e. $\eta \in L^1_{\rm{loc}}([0,+\infty[)$ with compact support.}
\end{equation}
Next, we denote by ${\mathrm supp}(\eta)$ the support of $\eta$. It is not restrictive to assume that ${\mathrm supp}(\eta)\subseteq(a,b)$ is a bounded open interval with $0\leq a<b$, and we consistently interchange ${\mathrm supp}(\eta)$ and $(a,b)$ throughout the text. We denote by $I_{p,{\mathrm supp}(\eta)}$ the set
\begin{align}
\label{decomposet}
I_{p,{\mathrm supp}(\eta)}:=\Big\{r\in({\mathrm supp}(\eta))^{\circ}:\,&\exists\,\epsilon >0\text{ such that }
\eta^{-\frac{1}{p-1}}\in L^1\left((r-\epsilon ,r+\epsilon )\right)\Big\},
\end{align}
where $({\mathrm supp}(\eta))^{\circ}$ denotes the interior of ${\mathrm supp}(\eta)$. The set $I_{p,{\mathrm supp}(\eta)}$ is the largest open set in $(a,b)$ such that $\eta^{-\frac{1}{p-1}}$ is locally summable. Without loss of generality, we can express $I_{p,{\mathrm supp}(\eta)}$ as the disjoint union
\begin{equation}\label{minassweight3}
I_{p,{\mathrm supp}(\eta)} = \bigcup_{i=1}^{{N_{p,\eta}}} (a_{p,i}, b_{p,i}),
\end{equation}
with $1\leq N_{p,\eta}\leq +\infty$. Subsequently, for the sake of a lean notation, we set $a_{i} \coloneqq a_{p,i}$, $b_{i} \coloneqq b_{p,i}$, $N_{\eta} \coloneqq N_{p,\eta}$. Let us now provide the following definition.
\begin{defi}\label{deff1}
\begin{itemize}
\item[(i)] If $I_{{\mathrm supp}(\eta),\eta}=\,\emptyset$, we put $N_\eta:=0$.
\item[(ii)]  If $1\le\,N_{\eta}<\,+\infty$  we say that $\eta$ is {\it finitely degenerate} in ${\mathrm supp}(\eta)$. 
\item[(iii)]   If $N_\eta=\,\infty$  we say that $\eta$ is {\it not finitely degenerate} in ${\mathrm supp}(\eta)$.
\end{itemize}
\end{defi}
Furthermore, we define the set
\begin{align}\label{Iomega}
I_{\Omega,w}\coloneqq \{x\in \Omega: \vert x\vert \in I_{p,{\mathrm supp}(\eta)}\},
\end{align}
and also
\begin{align}
\label{newset}
I_{\alpha_{i},\beta_{i}}\coloneqq \{x\in I_{\Omega,w}: \alpha_{i}<\vert x\vert< \beta_{i}\}, \hskip 0,2cm\, a_{i}<\alpha_{i}<\beta_{i}\leq b_{i}, \hskip 0,2cm i=1,\ldots, N_{\eta}.
\end{align}
\subsection{The class of the absolutely continuous functions in several variable}
By following \cite{Mal}, given a function $u:\Omega\rightarrow \mathbb{R}$, let us define the $d$-variation of $u$ on an open set $G\subset \Omega$ as
\begin{align}
V_{d}(u,G)\coloneqq \sup\left\{\sum_{i}\left({\rm Osc}_{B_{i}}u\right)^{d}: \text{$\{B_{i}\}$ is finite family of disjoint balls in $G$}\right\},
\end{align}
where ${\rm Osc}_{B_{i}}u$ denotes the oscillation of $u$ on $B_{i}$, and is defined as
\begin{align*}
{\rm Osc}_{B_{i}}u\coloneqq \sup_{x,y\in B_{i}}\vert u(x)-u(y) \vert.
\end{align*}
We say that $u$ has a $d$-bounded variation in $\Omega$ if $V_{d}(u,\Omega)<+\infty$. We denote by $BV^{d}(\Omega)$ the class of all functions with $d$-bounded variation in $\Omega$. Furthermore, we denote by $AC^{d}(\Omega)$ the space of all $d$-absolutely continuous functions in $BV^{d}(\Omega)$. Recall that a function $u:\Omega\rightarrow \mathbb{R}$ is $d$-absolutely continuous (see \cite[Page 2]{Mal}) if for each $\varepsilon>0$ there is $\delta=\delta(\varepsilon)>0$ such that for each disjoint finite family ${B_{i}}$ of closed balls in $\Omega$ we have
\begin{align*}
\sum_{i}\mathcal{L}^{d}(B_{i})\leq \delta \Rightarrow \sum_{i}\left({\displaystyle{\rm Osc}}_{B_{i}}u\right)^{d}<\varepsilon.
\end{align*} 

As proven in \cite[Theorem 3.3]{Mal} every function $u$ with $d$-bounded variation (and so every $d$-absolutely continuous function) is differentiable a.e. and its gradient $\nabla u\in L^d(\Omega;\R^d)$. Hence
\begin{equation}\label{sobolev}
AC^{d}(\Omega)\cap L^d(\Omega;\R^d)\subset W^{1,d}(\Omega).
\end{equation}
Moreover, if $p>d$, then
\begin{equation}\label{ifp>d}
W^{1,p}(\Omega)\subset AC^{d}(\Omega),
\end{equation}
(see \cite[Theorem 4.1]{Mal}).
Let $U\subset \Omega$ be an open bounded subset such that $I_{\Omega,w}\subset U$. We denote by $AC_{{\rm r}}^{d}(U)$ the following set of functions:
\begin{align}
\label{spazior}
\begin{aligned}
AC_{{\rm r}}^{d}(U)\coloneqq &\{u: U\rightarrow \mathbb{R}\, \text{measurable}: u\in W^{1,1}_{\rm{loc}}(I_{\Omega,w}),\\
&\phantom{formulaformula}\text{$u(x)=v(\vert x\vert)$ in $I_{\Omega,w}$ for some $v\in AC({\supp(\eta)})$}\}.
\end{aligned}
\end{align} 
\begin{lem}\label{bvdspace}
Let $AC_{r}^{d}(I_{\Omega,w})$ be defined as in \eqref{spazior}. Then
\begin{align}
AC_{r}^{d}(I_{\Omega,w})\subset AC^{d}(I_{\Omega,w}).
\end{align}
Furthermore, if $I_{\Omega,w}=\Omega$, then 
\begin{align}
AC_{r}^{d}(\Omega)\subset AC^{d}(\Omega).
\end{align}
\end{lem}
\begin{proof}
Let $u\in AC_{r}^{d}(I_{\Omega,w})$. Then by definition $u$ is radial and so there is $v\in AC({\supp(\eta)})$ such that $u(\cdot)=v(\vert \cdot\vert)$.   Furthermore, $v\in {\rm BV}({\supp(\eta)})$ the space of functions of bounded variation. Let us take $\varepsilon>0$, notice that
\begin{align}\label{mkmk}
{\rm Osc}_{B_{i}}u=\sup_{x,y\in B_{i}}\vert u(x)-u(y) \vert&=\sup_{x,y\in B_{i}}\vert v(\vert x\vert)-v(\vert y\vert) \vert
\leq \sup_{x,y\in B_{i}}{\rm ess} V_{\vert x\vert}^{\vert y\vert}(v),
\end{align}
where 
\begin{align*}
{\rm ess} V_{\vert x\vert}^{\vert y\vert}(v)\coloneqq \sup\left\{\sum_{j=1}^{m}\vert v(t_{j+1})-v(t_{j})\vert\right\}
\end{align*}
and the supremum taken over all finite partitions $\{|x|<t_{1}<\cdots <t_{m+1}< |y|\}$ such that $t_{j}$ is a point of continuity of $v$. Now consider $R_{i}$ be the radius of $B_{i}$.  Notice that 
\begin{align*}
\sum_{i=1}^{N}\mathcal{L}^{d}(B_{i})=\sum_{i=1}^{N}\frac{\pi^{\frac{d}{2}}}{\Gamma\left(\frac{d}{2}+1\right)}R_{i}^d \leq \delta,
\end{align*}
so that 
\begin{align*}
NC\max _{i=1,\ldots, N}R_{i}^{d} \leq \delta, \hskip 0,2cm C\coloneqq \frac{\pi^{\frac{d}{2}}}{\Gamma\left(\frac{d}{2}+1\right)},
\end{align*}
where $N$ denotes the number of balls.  Hence, we have that for each $i=1,\ldots,N$ that

\begin{align*}
R_{i}\leq \left(\frac{\delta}{NC}\right)^{\frac{1}{d}}.
\end{align*}
For simplicity, let us choose $\delta\coloneqq \frac{\varepsilon^{d}}{2^{d}C}$. Observe that 
\begin{align*}
\sum_{j=1}^{m}\vert t_{j+1}-t_{j}\vert \leq \vert \vert x\vert-\vert y \vert \vert\leq \vert x-y\vert \leq 2R_{i}\leq \frac{\varepsilon}{{N^{\frac{1}{d}}}}.
\end{align*}
The absolutely continuity property of $v$ implies that there exists $\delta_1>0$ such that
\begin{equation}\label{nmnmnm}
\sum_{j=1}^{m}\vert t_{j+1}-t_{j}\vert \leq \delta_1\  \Rightarrow \ 
\sum_{j=1}^{m}\vert v(t_{j+1})-v(t_{j})\vert<  \frac{\varepsilon}{{N^{\frac{1}{d}}}}.
\end{equation}
Now, if we choose $\varepsilon$ sufficiently small, we get
$$
\frac{\varepsilon}{(NC)^{\frac{1}{d}}}\leq \delta_1.
$$
Hence by \eqref{nmnmnm} for all finite partitions $\{|x|<t_{1}<\cdots <t_{m+1}< |y|\}$ such that $t_{i}$ is a point of continuity of $v$ it holds
$$
\sum_{j=1}^{m}\vert v(t_{j+1})-v(t_{j})\vert<  \frac{\varepsilon}{{N^{\frac{1}{d}}}}
$$
and, taking the supremum over these partitions, we obtain
$$
{\rm ess} V_{\vert x\vert}^{\vert y\vert}(v)\leq  \frac{\varepsilon}{{N^{\frac{1}{d}}}}.
$$
Thus by \eqref{mkmk}
\begin{align*}
{\rm Osc}_{B_{i}}u=\sup_{x,y\in B_{i}}\vert u(x)-u(y) \vert & \leq \frac{\varepsilon}{{N^{\frac{1}{d}}}}
\end{align*}
and we obtain that for every $i:1,\dots,N$
\begin{align*}
\sum_{i=1}^N\left({\displaystyle{\rm Osc}}_{B_{i}}u\right)^{d}\leq N \frac{\varepsilon^{d}}{N}\leq \varepsilon.
\end{align*}
Since $\varepsilon$ is arbitrary, we are done. 
\end{proof}
Let us consider
${\mathcal{W}}_{r}^{1,q}(\Omega)$ the Sobolev-type space of radial functions with $q\in [1,\infty]$, defined as
\begin{align*}
{\mathcal{W}}^{1,q}_r(\Omega):=\left\{u\in AC_{r}^{d}(\Omega): \nabla u\in L^{q}(\Omega)\right\}.
\end{align*} 
By \eqref{ifp>d}
if $p>d$, then
\begin{equation*}\label{ifp>d1}
W^{1,p}_r(\Omega)\subset AC_r^{d}(\Omega),
\end{equation*}
where
\begin{equation*}\label{W1p}
W^{1,p}_r(\Omega):=\{u\in W^{1,p}(\Omega): u {\text{ radial}}
\}.
\end{equation*}
Therefore
\begin{equation*}\label{ifp>d12}
W^{1,p}_r(\Omega)\subset {\mathcal{W}}^{1,p}_r(\Omega).
\end{equation*}
On the other hand, by \eqref{sobolev} for $p=d$
\begin{equation*}\label{sobolev1}
AC^{d}_r(\Omega)\cap L^d(\Omega)\subset W^{1,d}_r(\Omega),
\end{equation*}
then
$$
AC^{d}_r(\Omega)\subset L^d_{\rm{loc}}(\Omega),
$$
and we conclude that
\begin{equation*}\label{ifp>d123}
{\mathcal{W}}^{1,d}_r(\Omega)\subset W^{1,d}_{r,{\rm loc}}(\Omega).
\end{equation*}

\subsection{The degenerate functional}
Let us now define
\begin{equation*}
F(u)=\left\{
\begin{aligned}
& \int_{\Omega} |\nabla u|^p\,w\,\de x\text{   \ \ if } u\in AC_{r}^{d}(\Omega)\\
& +\infty \ \ \ \ \ \ \ \ \ \ \ \text{   \ \ if } u\in X\setminus AC_{r}^{d}(\Omega).
\end{aligned}
\right.
\end{equation*}
Here, $X$ is an appropriate set of integrable functions, that will be chosen in Section \ref{relaxp}. Further, let $\overline F:\,X\to [0,+\infty]$ denote the lower semicontinuous envelope of $F$ w.r.t. the topology of  $X$. As we will see later on, our objective is to characterize the relaxation of the functional  $F$ concerning $L^{p}(\Omega,(\hat{w}_{p})^{p-1})$-convergence, where $\hat{w}_{p}$ is is defined below (see \eqref{pesop}). Let us now set 
\begin{align*}
\label{unidim}
{\rm{Dom}}_{\eta}\coloneqq \left\{v: {\supp(\eta)}\rightarrow \R: v\in W_{{\rm loc}}^{1,1}(I_{p,\supp(\eta)}), \displaystyle\int_{I_{p,\supp(\eta)}}\hskip -1cmr^{d-1}\vert v'(r)\vert^{p}\eta(r)\de r<+\infty\right\}.
\end{align*}
We now introduce the set
\begin{equation}\label{domw}
{\rm{Dom}}_{{\mathrm r},w}:=\Big\{u:\Omega\to\R: \text{measurable and $u(x)=v(\vert x\vert)$ for some $v\in {\rm{Dom}}_{\eta}$}\Big\}\,.
\end{equation}
We point out that
\begin{equation*}\label{domw1}
{\rm{Dom}}_{{\mathrm r},w}=\Big\{u:\Omega\to\R: \text{measurable}, u\in W^{1,1}_{\rm{loc}}(I_{\Omega,w}), \int_{I_{\Omega,w}}|\nabla u|^p w\,\de x<+\infty
\Big\}\,.
\end{equation*}
\begin{lem}[Fundamental convergence]\label {lemma2} 
Suppose that $\eta$ is finitely degenerate. Let $(u_h)_h\subset W^{1,1}(\Omega)$ be a sequence of radial functions such that
\begin{enumerate}
\item[(a)] $\displaystyle\sup_{h\in\bbN}\int_{I_{\Omega,w}}|\nabla u_h|^{p}w\,\de x<+\infty$\,,
\item[(b)] for every $i=\,1,\dots,N_\eta$ there exists $c_i$ such that $a_i< c_i<b_i$ and there exist finite the following limits
$$
\lim_{h\to+\infty}u_h(x)=d_i\in \bbR\, \hskip 0,1cm \text{for all $x\in I_{\Omega,w}$ such that $\vert x\vert=c_{i}$}.
$$
\end{enumerate}
Then there exists a subsequence $(u_{h_k})$ and a radial function 
$u:\,I_{\Omega,w}\to\bbR$ such that 
\begin{enumerate}
\item[(i)] $\displaystyle\lim_{k\to+\infty}u_{h_k}(x) =u(x) $ for every $x\in I_{\Omega,w}$\,,
\item[(ii)] $u\in {\rm{Dom}}_{{\mathrm r},w}$,
\item[(iii)]
\begin{align*}
\displaystyle\int_{I_{\Omega,w}}\vert \nabla u \vert^{p}w\de x\leq \liminf_{h_{k}\rightarrow +\infty}\displaystyle\int_{I_{\Omega,w}}\vert \nabla u_{h_{k}}\vert^{p}w\de x.
\end{align*}
\end{enumerate}
\end{lem}
\begin{proof}
 Let us note that, by assumption (b), $I_{\Omega,w}\neq\emptyset$. By hypothesis, we have that $(u_{h})_{h}$ is a sequence of radial functions. That is,
\begin{align*}
 u_{h}(x)=v_{h}(\vert x\vert), \hskip 0,2cm \text{for all $h\in\mathbb{N}$, $x\in \Omega$.}
 \end{align*}
 Notice that 
\begin{align*}
\frac{\partial u_{h}(x)}{\partial x_{j}}=v_{h}'(\vert x\vert)\frac{x_{j}}{r}, \hskip 0,2cm\, r=\vert x \vert,\, j=1,\ldots,d,
\end{align*}
 and thus 
\begin{align*}
\int_{I_{\Omega,w}}|\nabla u_h|^{p}w\,\de x&=\int_{I_{\Omega,w}}\left[\left\{\sum_{j=1}^{d}|v_{h}'(\vert x\vert)|^{2}\frac{\vert x_{j}\vert^{2}}{\vert x\vert^{2}}\right\}^{\frac{1}{2}}\right]^{p}w\,\de x\\
&=\int_{I_{\Omega,w}}|v_{h}'(\vert x\vert)|^{p}w\,\de x.
\end{align*}
Moreover, since $w$ is radial, we have that by the change of variable theorem
\begin{align*}
\displaystyle\int_{I_{\Omega,w}}|v_{h}'(\vert x\vert)|^{p}w\,\de x&=\displaystyle\int_{I_{\Omega,w}}|v_{h}'(\vert x\vert)|^{p}\eta(\vert x\vert)\,\de x\\
&=\omega_{d}\displaystyle\int_{I_{p,\supp(\eta)}}\hskip -0,5cm|v_{h}'(r)|^{p}r^{d-1}\eta(r)\,\de r,
\end{align*}
where 
\begin{align}
\label{misurasfera}
\omega_{d}\coloneqq \mathcal{H}^{d-1}(\mathbb{S}^{d}).
\end{align}
By (a), notice that 
\begin{align*}
\sup_{h\in \mathbb{N}}\displaystyle\int_{I_{p,\supp(\eta)}}\hskip -0,5cm|v_{h}'(r)|^{p}r^{d-1}\eta(r)\,\de r<+\infty.
\end{align*}
Then there exist a subsequence $(v_{h_k})_k$ of $(v_h)_h$, 
and  a function $v\in L^{p}(I_{p,\supp(\eta)},\eta)$ such that 
\begin{equation}\label{(2)}
v_{h_k}'\to v {\ \ \rm weakly \ in \ }L^{p}(I_{p,\supp(\eta)},\eta)\text{ as }k\to\infty\,,
\end{equation} 
\begin{equation*}
\frac{\partial u_{h_{k}}}{\partial x_{j}}\to v(\vert x\vert)\frac{x_{j}}{\vert x\vert}{\ \ \rm weakly \ in \ }L^{p}(I_{\Omega,w},w)\text{ as }k\to\infty\,, j=1,\ldots,d.
\end{equation*} 
Moreover, let us observe that
\begin{equation*}\label{(3)}
L^{p}_{\rm{loc}}(I_{\Omega,w},w)\subset L^1_{\rm{loc}}(I_{\Omega,w})\,, \hskip 0,5cm L^{p}_{\rm{loc}}(I_{p,\supp(\eta)},\eta)\subset L^1_{\rm{loc}}(I_{p,\supp(\eta)}).
\end{equation*}
Indeed, let us notice that by the change of variable
\begin{align*}
\displaystyle\int_{I_{\Omega,w}} w^{-\frac{p'}{p}} \de x=\omega_{d}\displaystyle\int_{I_{p,\supp(\eta)}} \eta^{-\frac{p'}{p}}r^{d-1} \de r,
\end{align*}
where $\frac{1}{p}+\frac{1}{p'}=1$ . Then for $i=1,\ldots, N_{\eta}$, and for each $K\Subset (a_{i},b_{i})$, we get by \eqref{decomposet}, and since $\frac{p'}{p}=\frac{1}{p-1}$ that
\begin{align*}
\displaystyle\int_{K} \eta^{-\frac{p'}{p}}r^{d-1} \de r\leq b_{i}^{d-1}\displaystyle\int_{K} \eta^{-\frac{p'}{p}}\de r<+\infty.
\end{align*}
Let us take a compact set $\mathcal{K}\Subset I_{\alpha_{i},\beta_{i}}$, $i=1,\ldots, N_{\eta}$ with $I_{\alpha_{i},\beta_{i}}$ as defined in \eqref{newset}. Then by H\hol{o}lder's inequality
\begin{align}
\label{usefulbound}
\displaystyle\int_{\mathcal{K}}\vert  z\vert \de x\leq \left(\displaystyle\int_{\mathcal{K}}\vert  z\vert^{p}w\de x\right)^{\frac{1}{p}}\left(\displaystyle\int_{\mathcal{K}} w^{-\frac{p'}{p}} \de x\right)^{\frac{1}{p'}}<+\infty
\end{align}
for any $z\in L^{p}(\mathcal{K},w)$. Notice that from \eqref{(2)} and \eqref{usefulbound}, we get that $v\in L^1_{\rm{loc}}(I_{\alpha,\beta};\mathbb{R})$ and
\begin{equation*}\label{4}
\displaystyle\int_{I_{\alpha,\beta}} \left\vert \frac{\partial u_{h_k}}{\partial x_{j}}- v(\vert x\vert)\frac{x_{j}}{\vert x\vert}\right\vert \,\de x\rightarrow 0\, \text{ as }k\to\infty\,,
\end{equation*}
for each $[\alpha,\beta]\subset\,I_{p,{\mathrm supp}(\eta)}$.
Let us consider $u:\,\Omega\to\R$ defined in the following way:  let $i=\,1,\dots,N_\eta$, and consider an interval $(a_{i}, b_{i})$ and define $I_{a_{i},b_{i}}$ as in \eqref{newset}, so that we let
\begin{equation*}
u^{i}(x):=
\begin{cases}
\ \ \ \ \ 0 \ \ \ \ &\text{ if } x\in\Omega\setminus I_{a_{i},b_{i}},\,
\\
\displaystyle d_{i}+\int_{c_{i}}^{\vert x\vert}v(r)\de r&\text{ if } c_{i}\leq \vert x\vert<b_{i},\\
\displaystyle d_{i}+\int_{\vert x\vert}^{c_{i}}v(r)\de r&\text{ if } a_{i}\leq \vert x\vert<c_{i}.
\end{cases}
\end{equation*}
Then we define
$$
u(x)=\sum_{i=1}^{N_{\eta}} u^i(x)  \chi_{I_{a_{i},b_{i}}}(x) .
$$
In what follows, we aim to prove that
\begin{equation*}\label{aimed}
u\in W^{1,1}_{\text {loc}}(I_{\Omega, w}).
\end{equation*}
As before, let us take some $\mathcal{K}\Subset I_{\alpha_{i},\beta_{i}}$, $i=i,\ldots, N_{\eta}$. First, notice that by the fundamental theorem of calculus, one has that
\begin{align*}
\frac{\partial u}{\partial x_{j}}=v(\vert x\vert)\frac{x_{j}}{\vert x\vert},
\end{align*}
and then
\begin{align*}
\sum_{j=1}^{d}\displaystyle\int_{\mathcal{K}}\left\vert \frac{\partial u}{\partial x_{j}}\right\vert\de x\leq\displaystyle\int_{\mathcal{K}} \vert v(\vert x \vert)\vert \de x
&\leq \omega_{d}\displaystyle\int_{a_{i}}^{b_{i}}\hskip -0,2cm r^{d-1}v(r)\de r<+\infty,
\end{align*}
where $\omega_{d}$ is defined according \eqref{misurasfera}. Moreover, let us notice that 
\begin{align}\label{finitamente1}
\displaystyle \int_{I_{\Omega,w}}\vert \nabla u\vert^{p}w\de x=\displaystyle \int_{I_{\Omega,w}}\vert v'(\vert x\vert)\vert^{p}w\de x&=\omega_{d}\int_{I_{p,\supp(\eta)}}\vert v'(r)\vert^{p} r^{d-1}\eta(r) \de r\\
&\leq \omega_{d}\liminf_{h_{k}}\displaystyle\int_{I_{p,\supp(\eta)}}\vert v_{h_{k}}'(r)\vert^{p} r^{d-1}\eta(r) \de r\nonumber\\
&=\liminf_{h_{k}}\displaystyle\int_{I_{\Omega,w}}\vert \nabla u_{h_{k}}\vert^{p}w\de x<+\infty,\nonumber
\end{align}
where in \eqref{finitamente1} we have used that $\eta$ is finitely degenerate, and in the last inequality we have used the radiality and assumption (a).
\end{proof}
\subsection{An auxiliary weight}
In what follows, by following closely \cite{CCMH1} we define a suitable weight $\hat{w}_{p}$ for $1<p<+\infty$ for which it is possible to prove a Poincar\'e inequality involving $w$ and $(\hat{w}_{p})^{p-1}$.
Let $w:\,\Omega\to [0,\infty)$ be a function satisfying \eqref{minassweight} and \eqref{minassweight3}. We let $\wstar_{p}:\Omega\to [0,+\infty[$ be defined as
\begin{equation}\label{pesop}
\wstar_{p}(x):=
\begin{cases}
\displaystyle\lim_{\vert x\vert\to a_i^+}  \left(\int_{I_{\vert x\vert,\frac{a_{i}+b_{i}}{2}}}\frac{1}{[\vert y\vert^{(d-1)p} w(y)]^{\frac{1}{p-1}}}\,\de y\right)^{-1}    &\text{ if } \vert x\vert=a_i \\
\left(\displaystyle\int_{I_{\vert x\vert,\frac{a_{i}+b_{i}}{2}}}\frac{1}{[\vert y\vert^{(d-1)p} w(y)]^{\frac{1}{p-1}}}\,\de y\right)^{-1} &\text{ if } {a_i}< \vert x\vert \leq {{3a_i+b_i}\over{4}}\\
\left(\displaystyle\int_{I_{\frac{3a_{i}+b_{i}}{4},\frac{3a_{i}+b_{i}}{4}}}\frac{1}{[\vert y\vert^{(d-1)p} w(y)]^{\frac{1}{p-1}}}\,\de y\right)^{-1} &\text{ if } {{3a_i+b_i}\over{4}}\leq \vert x\vert\leq {{a_i+3b_i}\over{4}}\\
\left(\displaystyle\int_{I_{\frac{a_{i}+b_{i}}{2},\vert x\vert }}\frac{1}{[\vert y\vert^{(d-1)p} w(y)]^{\frac{1}{p-1}}}\,\de y\right)^{-1} &\text{ if } {{a_i+3b_i}\over{4}}\leq \vert x\vert<b_i\\
\displaystyle\lim_{\vert x\vert\to b_i^-}  \left(\displaystyle\int_{I_{\frac{a_{i}+b_{i}}{2},\vert x\vert}}\frac{1}{[\vert y\vert^{(d-1)p} w(y)]^{\frac{1}{p-1
}}}\,\de y\right)^{-1}    &\text{ if } \vert x\vert=b_i\\
\ \ \ \ \ \ \ \ \ 0&\text{ if } x\in \Omega\setminus \overline{I_{\Omega,w}}\,.
\end{cases}
\end{equation}
\begin{oss}\label{newoss}
Let us point out that the definition of $\hat{w}_{p}$ is the extension of the one considered in \cite{CCMH1} to our multivariate context. Like as in the one-dimensional case, its definition heavily depends on $p$, and it is defined as the inverse of an multivariate integral term along annular regions (or spherical shells) dictated by the decomposition of the weight $\eta$. This  allows us to use its nice properties, such as continuity, that  is needed to prove Proposition \ref{propw}. Furthermore, let us also notice that $\wstar_{p}$ still is a radial function as the original weight $w$. That is,
\begin{align*}
\label{relation}
\wstar_{p}(\cdot)=\omega_{d}^{-1}\hat{\eta}_{p}(\vert \cdot\vert),
\end{align*}
where
\begin{equation*}
\hat{\eta}_{p}(t):=
\begin{cases}
\displaystyle\lim_{t\to a_i^+}  \left(\int_{t}^{\frac{a_{i}+b_{i}}{2}}\frac{1}{(s^{d-1}\eta(s))^{\frac{1}{p-1}}}\,\de s\right)^{-1}    &\text{ if } t=a_i \\
\left(\displaystyle\int_{t}^{\frac{a_{i}+b_{i}}{2}}\frac{1}{(s^{d-1}\eta(s))^{\frac{1}{p-1}}}\,\de s\right)^{-1} &\text{ if } {a_i}< t \leq {{3a_i+b_i}\over{4}}\\
\left(\displaystyle\int_{\frac{3a_{i}+b_{i}}{4}}^{\frac{3a_{i}+b_{i}}{4}}\frac{1}{(s^{d-1}\eta(s))^{\frac{1}{p-1}}}\,\de s\right)^{-1} &\text{ if } {{3a_i+b_i}\over{4}}\leq t\leq {{a_i+3b_i}\over{4}}\\
\left(\displaystyle\int_{\frac{a_{i}+b_{i}}{2}}^{t}\frac{1}{(s^{d-1}\eta(s))^{\frac{1}{p-1}}}\,\de s\right)^{-1} &\text{ if } {{a_i+3b_i}\over{4}}\leq t<b_i\\
\displaystyle\lim_{t\to b_i^-}  \left(\displaystyle\int_{\frac{a_{i}+b_{i}}{2}}^{t}\frac{1}{(s^{d-1}\eta(s))^{\frac{1}{p-1}}}\,\de s\right)^{-1}    &\text{ if } t=b_i\\
\ \ \ \ \ \ \ \ \ 0&\text{ if } t\in \supp(\eta)\setminus \overline{I_{p,\supp(\eta)}}\,.
\end{cases}
\end{equation*}
\end{oss}
\begin{oss}\label{serra}
As in \cite[Sect. 4.2]{CC} we could introduce a truncated weight
\begin{equation*}\label{tildeeta}
\tilde \eta_p(r):=\,\min\{\eta(r),\,\hat{\eta}_{p}(r),\,1\}\,, \ \ \ r\in [0,+\infty]
\end{equation*}
if $r\in\Omega$ is a Lebesgue's point of $\eta$ and $0$ otherwise, and the corresponding weight
\begin{equation*}\label{tildew}
\tilde w_p(x):=\tilde \eta_p(|x|)=\min\{w(x),\,\wstar_p(x),\,1\}.
\end{equation*}
As in \cite[Proposition 4.6]{CC} $\tildew\in L^\infty(\Omega)$ and
\begin{equation}\label{incltildew}
L^2(\Omega,\wstar_p)\cup L^2(\Omega,w)\cup L^2(\Omega)\subset L^2(\Omega,\tildew_p)\,.
\end{equation}
\end{oss}
In the next Lemma we prove that \eqref{pesop} is well-defined. We only need to observe that
\begin{align}
\label{equivalenza}
I_{p,\supp(r^{d-1}\eta(\cdot))}=I_{p,\supp(\eta)}.
\end{align}
\begin{lem}
Let us consider the $I_{p,\supp(\eta)}$, and $I_{p,\supp(r^{d-1}\eta(\cdot))}$ defined according to \eqref{decomposet}. Then \eqref{equivalenza} holds true.
\end{lem}
\begin{proof}
Let us first prove that $I_{p,\supp(\eta)}\subset I_{p,\supp(r^{d-1}\eta(\cdot))}$. Take $r\in I_{p,\supp(\eta)}$. Then there exists $\varepsilon>0$ such that $\eta^{-\frac{1}{p-1}}\in L^{1}(r-\varepsilon,r+\varepsilon)$. Notice that
\begin{align*}
\displaystyle\int_{r-\varepsilon}^{r+\varepsilon}\frac{1}{[s^{d-1}\eta(s)]^{\frac{1}{p-1}}}\de s\leq \frac{1}{[r-\varepsilon]^{\frac{d-1}{p-1}}}\displaystyle\int_{r-\varepsilon}^{r+\varepsilon}\frac{1}{[\eta(s)]^{\frac{1}{p-1}}}\de s<+\infty
\end{align*}
and thus $r\in I_{p,\supp(r^{d-1}\eta(\cdot))}$. Let us prove the reverse inclusion. Let us take $r\in I_{p,\supp(r^{d-1}\eta(\cdot))}$. Notice that
\begin{align*}
\displaystyle\int_{r-\varepsilon}^{r+\varepsilon}\frac{1}{[\eta(s)]^{\frac{1}{p-1}}}\de s&\leq (r+\varepsilon)^{\frac{d-1}{p-1}}\displaystyle\int_{r-\varepsilon}^{r+\varepsilon}\frac{1}{[s^{d-1}\eta(s)]^{\frac{1}{p-1}}}\de s<+\infty,
\end{align*}
and we are done.
\end{proof}
The previous function $\hat{w}_{p}$ will play an important role in the relaxation of $F$. In particular, we will consider its relaxation involving the $L^{p}(\Omega, (\hat{w}_{p})^{p-1})$-convergence. Before providing the precise details of how we relax $F$, let us before reproduce  some properties of the functions $\wstar_{p}$ in the following proposition closely to the one obtained in \cite[Proposition 2.5]{CCMH1}.  To this aim, we need to introduce the following notion of increasing, and decreasing functions along curves.
\begin{defi}
Let us suppose that we have a curve $\gamma:[0,1]\rightarrow \R^{d}$.
\begin{enumerate}
\item[I.]We say that a function $u:\R^{d}\rightarrow \R$ is increasing along $\gamma$ if the following holds true.
\begin{align}
\text{If $0\leq t_{1}<t_{2}\leq 1$, and $\vert \gamma(t_{1})\vert \leq \vert \gamma(t_{2})\vert$ implies that $u(\gamma(t_{1}))\leq u(\gamma(t_{2}))$.}
\end{align}
\item[II.] We say that a function $u:\R^{d}\rightarrow \R$ is decreasing along $\gamma$ if the following holds true.
\begin{align}
\text{If $0\leq t_{1}<t_{2}\leq 1$, and $\vert \gamma(t_{1})\vert \leq \vert \gamma(t_{2})\vert$ implies that $u(\gamma(t_{2}))\leq u(\gamma(t_{1}))$.}
\end{align}
\end{enumerate}
\end{defi}
\begin{prop}\label{propw}
\begin{itemize}\
\item[(i)] Suppose that $w^{-\frac{1}{p-1}}$ is  not locally summable in $\Omega$, that is, $I_{\Omega,w}=\,\emptyset$. Then $\wstar_{p}\equiv 0$.
\item[(ii)] For each $i=1,\dots, N_\eta$, let us define $I_{\alpha_{i},\beta_{i}}$ $\alpha_{i}<\beta_{i}$ as in \eqref{newset} for two generic numbers $\alpha_{i},\beta_{i}$. $\wstar_{p}$ is constant in $I_{{{3a_i+b_i}\over{4}},{{a_i+3b_i}\over{4}}}$, increasing along the curve 
$\gamma_{1}(t)=(1-t)x+ty$ which is contained in $I_{{a_i},{{3a_i+b_i}\over{4}}}$, and  where $\vert x\vert=a_{i}$, $\vert y\vert=\frac{3a_i+b_i}{4}$, and $t\in [0,1]$, decreasing along the curve $\gamma_{2}(t)=(1-t)x+ty$, where $\vert x \vert=\frac{a_i+3b_i}{4}$, $\vert y\vert=b_{i}$ which is contained in $I_{{{a_i+3b_i}\over{4}},b_i}$ and absolutely continuous along $\gamma_{1}(t)$ and $\gamma_{2}(t)$, $t\in [0,1]$.  In particular, the following hold true:
\[
0\,<\wstar_{p}(x)\le\,\sup_{y\in I_{a_i,b_i}}\wstar_{p}(y)<\,\infty\quad\forall\,  x \in I_{a_i,b_i}\,,
\]
\[
\inf_{x\in I_{\alpha,\beta}}w(x)>\,0\text{ for each $x\in I_{\alpha,\beta}$, $a_i<\,\alpha<\,\beta<\,b_i$, }
\]
and $\wstar_{p}(a_i)=\,0$ (respectively $\wstar_{p}(b_i)=\,0$) if and only if $w^{-\frac{1}{p-1}}\notin L^1\left(I_{a_i,\frac{a_i+b_i}{2}}\right)$ (respectively $w^{-\frac{1}{p-1}}\notin L^1\left(I_{\frac{a_i+b_i}{2},b_i}\right)$). 
\item[(iii)] We have
\begin{equation*}\label{derwstar}
\frac{\partial \wstar_{p}}{\partial x_{j}}
={{\left(\wstar_{p}\right)^2}\over{w^{\frac{1}{p-1}}}}\frac{x_{j}}{\vert x\vert}\quad\text{ a.e. in }I_{{a_i},{{3a_i+b_i}\over{4}}}\cup I_{{{a_i+3b_i}\over{4}},b_i}\,.
\end{equation*}
\item[(iv)] Suppose that $w^{-\frac{1}{p-1}}\in L^1(\Omega)$. Then there exists a  constant $c>\,0$ such that
\[
0<\,\frac{1}{c}\le\,\wstar_{p}(x)\le\,c\quad\text{ a.e. } x\in\Omega\,.
\]
\item[(v)] Suppose that $w$ is finitely degenerate in $\Omega$, that is, \eqref{minassweight3} holds with $1\le\,N_\eta<\,\infty$. Then there exists a  constant $c>\,0$ such that
\[
0\le\,\wstar_{p}(x)\le\,c\quad\text{ a.e. }x\in\Omega\,.
\]
\item[(vi)] Suppose that $w$ is not finitely degenerate in $\Omega$, that is, \eqref{minassweight3} holds with $N_\eta=\,\infty$. Then $\wstar_{p}\in L^\infty_{\rm loc}(I_{\Omega,w})$.
\end{itemize}
\end{prop}
\begin{oss}
Let us notice that since $\wstar_{p}$ is a radial function, along every increasing and  open curve $\gamma: [0,1]\rightarrow I_{a_{i},b_{i}}$, $i=1,\ldots, N_{\eta}$, the statements of Proposition \autoref{propw} hold true.  Furthermore, as already pointed out in the onedimensional (see \cite[Remark 2.6]{CCMH1}), if $w$ is not finitely degenerate in some open set $\Omega$, then it can also happens that $\wstar_{p}\notin L^{1}(\Omega)$. Here, the example given in \cite[Remark 2.6]{CCMH1} serves as a recipient to the present multidimensional case. Indeed, suppose that $\supp(\eta)\subset (0,1)$ and let $(a_{i},b_{i})$, $i=1,\ldots, +\infty,$ be a sequence of disjoint open intervals in $(0,1)$ and $m_{i}$ be a sequence of positive real numbers which will be fixed later on. Let $\eta:(0,1)\rightarrow [0,+\infty)$ defined as follows:
\begin{align*}
\eta(r)\coloneqq
\begin{cases}
m_{i}r^{1-d}(r-a_{i})^{\alpha}& \hskip 0,1cm \text{if $a_{i}\leq r\leq \frac{a_{i}+b_{i}}{2}$,}\\
m_{i}r^{1-d}(b_{i}-r)^{\alpha}& \hskip 0,1cm \text{if $\frac{a_{i}+b_{i}}{2}\leq r\leq b_{i}$,}\\
0&\hskip 0,1cm \text{otherwise}.
\end{cases}
\end{align*}
Let us fix $a_{i}\leq r\leq \frac{3a_{i}+b_{i}}{4}$. Then the corresponding auxiliar weight associated to $\eta$ is given by 
\begin{align*}
\hat{\eta}_{p}(r)=\frac{(\alpha_{p}-1)m_{i}^{\frac{1}{p-1}}(r-a_{i})^{\alpha_{p}-1}}{1-\left(\frac{2(r-a_{i})}{b_{i}-a_{i}}\right)^{\alpha_{p}-1}},
\end{align*}
where $\alpha_{p}\coloneqq \frac{\alpha}{p-1}$, and thus $\wstar_{p}(x)=\omega_{d}^{-1}\hat{\eta}_{p}(\vert x\vert)$, $x\in B_{1}(0)$ the $d$-dimensional ball of radius $1$. By \cite[Remark 2.6]{CCMH1}, one gets that $\hat{\eta}_{p}\notin L^{1}((0,1))$, and thus $\wstar_{p}\notin L^{1}(\Omega)$.
\end{oss}
\section{Weighted Poincar\'e inequalities}\label{sec:3}
\subsection{A Poincar\'e inequality with a double weight}
In what follows, we derive a weighted Poincar\'e inequality that we use later on. We first state some preliminary lemmas.
\begin{prop}\label {poincarep}
Let $d\geq 1$ be a natural number and define $\omega_{d}$ as in \eqref{misurasfera}. Let us consider a fixed $u\in {\rm{Dom}}_{{\mathrm r},w}$, $i=1,\dots,N_\eta$, and let $\frac{1}{p}+\frac{1}{p\prime}=1$. Let us take $\zeta,x\in I_{\Omega,w}$ such that $a_i<\vert\zeta\vert\leq \vert x\vert\leq {{a_i+b_i}\over{2}}$. The following hold true:
\begin{equation}\label{a1}
|u(x)-u(\zeta)|\omega_{d}\,\sqrt[p\prime]{\wstar_{p}(\zeta)}
\leq \left(\int_{I_{\vert \zeta\vert,\vert x\vert}}|\nabla u(y)|^{p}w(y)\,\de y\right)^{1\over p}\,;
\end{equation}
\begin{equation}\label{a2}
|u(\zeta)|^{p}(\wstar_{p}(\zeta))^{p-1}\omega_{d}^{p}\leq 2^{p-1}\left(|u(x)|^{p}(\wstar_{p}(\zeta))^{p-1}\omega_{d}^{p}+\int_{I_{a_{i},\vert x \vert}}|\nabla u(y)|^{p}w(y)\,\de y\,\right).
\end{equation}
Let us take $\zeta,x$ such that ${{a_i+b_i}\over{2}}\leq \vert x\vert\leq \vert\zeta\vert<b_i$. The following hold true:
\begin{equation}\label{a3}
|u(x)-u(\zeta)|\omega_{d}\,\sqrt[p\prime]{\wstar_{p}(\zeta)}
\leq \left(\int_{I_{\vert x\vert,\vert \zeta\vert}}|\nabla u(y)|^{p}w(y)\,\de y\right)^{1\over {p}}\,;
\end{equation}
\begin{equation}\label{a4}
|u(\zeta)|^{p}\omega_{d}^{p}(\wstar_{p}(\zeta))^{p-1}\leq 2^{p-1}\left(|u(x)|^{p}\omega_{d}^{p}(\wstar_{p}(\zeta))^{p-1}+\int_{I_{\vert x\vert,b_i}}|\nabla u(y)|^{p}w(y)\,\de y\,\right).
\end{equation}
\end{prop}
\proof 
In what follows, we closely follows the proof of \cite[Proposition 2.8]{CCMH1}. Let us consider $u\in {\rm{Dom}}_{{\mathrm r},w}$. By definition, one has that there exists $v\in {\rm{Dom}}_{\eta}$ such that 
\begin{align*}
u(x)=v(\vert x\vert) \text{ for a.e. $x\in I_{\Omega,w}, v\in W^{1,1}(I_{p,\supp(\eta)})$.}
\end{align*}
By the immersion of $W^{1,1}(I_{p,\supp(\eta)})$ into $AC(I_{p,\supp(\eta)})$ , we also have that $v\in AC_{\rm loc}((a_i,b_i))$, for all $i=1,\ldots, N_{w}$. Then for every $r_{1}, r_{2}\in]a_i,{{a_i+b_i}\over{2}}]$ such that $a_i<r_{2}\leq r_{1}\leq {{a_i+b_i}\over{2}}$ we have
\begin{equation*}
\begin{split}
|v(r_{1})-v(r_{2})|=&\left|\int_{r_{2}}^{r_{1}}v'(r)\,\de r\right|.
\end{split}
\end{equation*}
Notice that
\begin{equation*}
\begin{split}
|v(r_{1})-v(r_{2})|=&\left|\int_{r_{2}}^{r_{1}}v'(r)\,\de r\right|\leq
\left(\omega_{d}\int_{r_{2}}^{r_{1}}|v'(r)|^{p}r^{d-1}\eta(r)\,\de r\right)^{1\over p}
\left(\omega_{d}^{-\frac{p'}{p}}\int_{r_{2}}^{r_{1}} [r^{d-1}\eta]^{-\frac{p'}{p}}(r)\,\de r\right)^{1\over p'}\\
&  
\leq \left(\omega_{d}\int_{r_{2}}^{r_{1}}|v'(r)|^{p}r^{d-1}\eta(r)\,\de r\right)^{1\over p}
\left(\omega_{d}^{-\frac{p'}{p}}\int_{r_{2}}^{{{a_i+b_i}\over{2}}}[r^{d-1}\eta]^{-\frac{p'}{p}}(r)\,\de r\right)^{1\over p'}\,.
\end{split}
\end{equation*}
Then, by letting $\vert x\vert=r_{1}$, and $r_{2}=\vert \zeta\vert$ with $x,\zeta\in I_{a_{i},\frac{a_{i}+b_{i}}{2}}$, one gets
\begin{align*}
 \omega_{d}\displaystyle\int_{\vert\zeta\vert}^{\vert x\vert} \vert v'(r)\vert^{p}r^{d-1}\eta(r) \de r
&=\omega_{d}\displaystyle\int_{\vert\zeta\vert}^{\vert x\vert}\left\vert \frac{ v'(r)\sqrt{\sum_{j=1}^{d}\vert x_{j}\vert^{2}}}{r}\right\vert^{p}r^{d-1}\eta(r)\de r\\
&=\omega_{d}\displaystyle\int_{\vert\zeta\vert}^{\vert x\vert}\left\vert \frac{\sqrt{\sum_{j=1}^{d}\vert v'(r)x_{j}\vert^{2}}}{r}\right\vert^{p}r^{d-1}\eta(r)\de r\\
&=\omega_{d}\displaystyle\int_{\vert\zeta\vert}^{\vert x\vert}\left\vert \sqrt{\sum_{j=1}^{d}\left\vert \frac{\partial u}{\partial x_{j}}\right\vert^{2}}\right\vert^{p}r^{d-1}\eta(r)\de r\\
&=\displaystyle\int_{I_{\vert \zeta\vert,\vert x\vert}}\vert \nabla u\vert^{p}w\de y.
\end{align*}
Hence, we have obtained that
\begin{align}
\nonumber
|u(x)-u(\zeta)|&\leq \left(\displaystyle\int_{I_{\vert \zeta\vert,\vert x\vert}}\vert \nabla u\vert^{p}w\de y\right)^{\frac{1}{p}}\left(\omega_{d}^{-\frac{p'}{p}}\int_{\vert \zeta\vert}^{{{a_i+b_i}\over{2}}}[r^{d-1}\eta]^{-\frac{p'}{p}}(r)\,\de r\right)^{1\over p'}\\
&=\left(\displaystyle\int_{I_{\vert \zeta\vert,\vert x\vert}}\vert \nabla u\vert^{p}w\de y\right)^{\frac{1}{p}}\left(\omega_{d}^{1-p'}\int_{\vert \zeta\vert}^{{{a_i+b_i}\over{2}}}[r^{d-1}\eta]^{-\frac{p'}{p}}(r)\,\de r\right)^{1\over p'}\label{a5}\\
&=\left(\displaystyle\int_{I_{\vert \zeta\vert,\vert x\vert}}\vert \nabla u\vert^{p}w\de y\right)^{\frac{1}{p}}\left(\omega_{d}^{1-p'}\int_{\vert \zeta\vert}^{{{a_i+b_i}\over{2}}}r^{-(d-1)p'}r^{d-1}\eta^{-\frac{p'}{p}}(r)\,\de r\right)^{1\over p'}\nonumber\\
&=\left(\displaystyle\int_{I_{\vert \zeta\vert,\vert x\vert}}\vert \nabla u\vert^{p}w\de y\right)^{\frac{1}{p}}\left(\omega_{d}^{1-p'}\int_{I_{\vert \zeta\vert,{{a_i+b_i}\over{2}}}}[\vert y\vert^{(d-1)p} w(y)]^{-\frac{1}{p-1}}\,\de y\right)^{1\over p'}\nonumber,
\end{align}
where in equality we have used that $\frac{p'}{p}+1=p'$. Furthermore, notice that if $a_i<\vert\zeta\vert\le\min\{\frac{3a_i+b_i}4,\vert x\vert\}$, then \eqref{a1} follows by \eqref{a5} and the definition of $\wstar$. Furthermore, if $\frac{3a_i+b_i}4\le \vert \zeta\vert\le\,\frac{a_i+b_i}2$, since we have that
\[
\left(\int_{\vert\zeta\vert,{{a_i+b_i}\over{2}}}[\vert y\vert^{(d-1)p} w(y)]^{-\frac{1}{p-1}}\,\de y\right)^{1\over p\prime}\le\,\frac{1}{\sqrt[p\prime]{\wstar_{p}(\zeta)}}\,,
\]
\eqref{a1} still follows by \eqref{a5} and the definition of $\wstar_{p}$.
Then, since
\begin{equation*}
|u(\zeta)|^{p}\leq2^{p-1}\left(|u(x)|^{p}+|u(\zeta)-u(x)|^{p}\right)\,,
\end{equation*}
by \eqref{a1}, we then deduce \eqref{a2}. The remaining formulas \eqref{a3} and (\ref{a4}) follow  by arguing in a similar way.
\qed
\begin{oss}
Let us observe that the previous proposition is the multidimensional version of \cite[Proposition 2.8]{CCMH1}. Here, we have closely followed the argument used in the one-dimensional case, which relies on the fundamental theorem of calculus. However, the multivariate version of this theorem is somewhat different and is not directly related to absolutely continuous functions. Additionally, unlike the one-dimensional case, in our current setting we have a Poincaré inequality, which differs from its onedimensional counterpart by a factor of  $\omega_{d}^{p}$.
\end{oss}
Let us now give some consequences of Proposition \ref{poincarep} in the following Corollary.
\begin{cor}\label{proppp1}
Let us fix $u\in {\rm{Dom}}_{{\mathrm r},w}$, and $i=1,\dots,N_\eta$. Then the following hold true:
\begin{itemize}
\item[(i)]
$\displaystyle
|u(\zeta)|^{p}\omega_{d}^{p}(\wstar_{p}(\zeta))^{p-1}
\leq 2^{p-1}\left(\left|u\left(x\right)\right|^{p}\omega_{d}^{p}(\wstar_{p}(\zeta))^{p-1}
+\int_{I_{a_{i},b_{i}}}|\nabla u(y)|^{p}w(y)\,\de y\right)\,,
$
for each $\zeta\in I_{\Omega,w}$ with $\vert\zeta\vert\in (a_i,b_i)$, and for each $x\in I_{\Omega,w}$, such that $\vert x \vert=\frac{a_i+b_i}{2}$. Furthermore, $u\in L^{p}(I_{a_{i},b_{i}},(\wstar_{p})^{p-1})$, and if $N_{\eta}<+\infty$  then $u\in L^{p}(\Omega,(\wstar_{p})^{p-1})$.
\item[(ii)] Let us suppose that
\begin{align*}
\displaystyle{\int_{a_i}^{\frac{a_i+b_i}{2}}\frac{1}{[r^{d-1}\eta(r)]^{\frac{1}{p-1}}}\de r=+\infty}
\end{align*}
(respectively, suppose that $\displaystyle{\int_{\frac{a_i+b_i}{2}}^{b_i}\frac{1}{[r^{d-1}\eta(r)]^{\frac{1}{p-1}}}\de r=+\infty}$). Then there exists 
$$\displaystyle\lim_{\vert x\vert\to a_i^+} (u^{p} \,(\wstar_{p})^{p-1})(x)=\,0 \,(\text{respectively,}\lim_{\vert x\vert\to b_i^-} (u^{p} \,(\wstar_{p})^{p-1})(x)=\,0)\,.
$$
\item[(iii)] Suppose that 
\begin{align*}
\displaystyle{\int_{a_i}^{\frac{a_i+b_i}{2}}\frac{1}{[r^{d-1}\eta(r)]^{\frac{1}{p-1}}}\de r<\,\infty}
\end{align*}
(respectively, suppose that $\displaystyle{\int_{\frac{a_i+b_i}{2}}^{b_i}\frac{1}{[r^{d-1}\eta(r)]^{\frac{1}{p-1}}}\de x<\,\infty}$). Then 
\[u\in AC_{r}^{d}\Big(I_{a_i,\frac{a_i+b_i}{2}}\Big) \,(\text{respectively, }u\in AC_{r}^{d}\Big(I_{\frac{a_i+b_i}{2},b_i}\Big)\,.
\]
\end{itemize}
\end{cor}
\proof
(i) Note that by (\ref{a1}) and  (\ref{a2}) with $x$ such that $\vert x\vert={{a_i+b_i}\over{2}}$, we can obtain the desired inequality. Let us now justify (ii).  Consider $\zeta,x\in I_{\Omega,w}$ such that $ a_i<\vert\zeta\vert\leq \vert x\vert\leq \frac{a_i+b_i}{2}$. Further, suppose that $\displaystyle{\int_{a_i}^{\frac{a_i+b_i}{2}}\frac{1}{[r^{d-1}\eta(r)]^{\frac{1}{p-1}}}\de r=+\infty}$. By the definition of $\wstar_{p}$ and its radiality, we obtain that $\lim_{\vert\zeta\vert\to a_i^+}\wstar_{p}(\zeta)=0$.
Furthermore,  for each $x\in I_{a_i,\frac{a_i+b_i}{2}}$, we have that by \eqref{a2} the following inequality holds true:
$$
\limsup_{\vert\zeta\vert\to a_i^+}|u(\eta)|^{p}\omega_{d}^{p}(\wstar_{p}(\eta))^{p-1}\leq 2^{p-1}\int_{I_{a_{i},\vert x\vert}}|\nabla u(y)|^{p}w\,\de y.
$$
Hence, by letting $\lim$ as $\vert x\vert\to a_i^+$ in the previous inequality, then
\[
\lim_{\vert \zeta\vert\to a_i^+}|u(\zeta)|^{p}(\wstar_{p}(\zeta))^{p-1}= 0\,.
\]
The same reasoning works for $\displaystyle{\int_{\frac{a_i+b_i}{2}}^{b_i}\frac{1}{[r^{d-1}\eta(r)]^{\frac{1}{p-1}}}\de x=+\infty}$ because the radiality of $\hat{w}_{p}$. Then, we immediately obtain that
\[
\lim_{\vert \zeta\vert\to b_i^-}|u(\zeta)|^{p}(\wstar_{p}(\zeta))^{p-1}= 0\,.
\]
(iii) To conclude, let us now suppose that $\displaystyle{\int_{a_i}^{\frac{a_i+b_i}{2}}\frac{1}{[r^{d-1}\eta(r)]^{\frac{1}{p-1}}}\de r<\,\infty}$. We now prove that $u\in AC_{r}^{d}\big(I_{a_i,\frac{a_i+b_i}{2}}\big)$. Since $u\in AC_{r}^{d}\big(I_{a_i+\delta,\frac{a_i+b_i}{2}}\big)$, for each $\delta>0$, it is sufficient to prove that there exists the following limit
\begin{equation}\label{popol}
\lim_{\vert \zeta\vert\to a_i^+}u(\zeta)\in\R.
\end{equation}
Indeed, since $u$ can be written in terms of a radial function $v$, one gets the following: consider $\zeta\in I_{a_i,\frac{a_i+b_i}{2}}$. Since $u(\cdot)=v(\vert \cdot\vert)$
\begin{equation}\label{mbmb}
v(\vert \zeta\vert)=v\Big(\frac{a_i+b_i}{2}\Big)-\int_{\vert\zeta\vert}^{\frac{a_i+b_i}{2}}v'(r) \de r,
\end{equation}
for some $x\in I_{\Omega,w}$ such that $\vert x\vert=\frac{a_i+b_i}{2}$. Furthermore, let us notice that
\begin{align}
\int_{\vert \zeta\vert}^{\frac{a_i+b_i}{2}}\vert v'(r)\vert \de r\leq &\left(\omega_{d}\int_{\vert \zeta\vert}^{\frac{a_i+b_i}{2}}\vert v'(r)\vert^{p}r^{d-1}\eta(r)\de r\right)^{\frac{1}{p}}\left(\omega_{d}^{-\frac{1}{p-1}}\int_{\vert \zeta\vert}^{\frac{a_i+b_i}{2}}(r^{d-1}\eta(r))^{-\frac{1}{p-1}}\de r\right)^{\frac{1}{p'}}\nonumber\\
\label{rtrt}
&\leq \left(\left(\frac{a_{i}+b_{i}}{2}\right)^{d-1}\omega_{d}\int_{\vert \zeta\vert}^{\frac{a_i+b_i}{2}}\vert v'(r)\vert^{p}\eta(r)\de r\right)^{\frac{1}{p}}\times\\
&\phantom{formulaformula}\times\left(\omega_{d}^{-\frac{1}{p-1}}\int_{\vert \zeta\vert}^{\frac{a_i+b_i}{2}}(r^{d-1}\eta(r))^{-\frac{1}{p-1}}\de r\right)^{\frac{1}{p'}}
<+\infty\nonumber.
\end{align}
Therefore, by \eqref{mbmb} and \eqref{rtrt} ,we then deduce the existence of the desired limit \eqref{popol}. Lastly, let us notice that the remaining case follows by using the previous reasoning.
\qed
In what follows, we state our Poincar\'e type inequality in higher dimension with respect to the weight function $(\hat{w}_{p})^{p-1}$.
\begin{teo}[Poincar\'e type inequality on ${\rm{Dom}}_{{\mathrm r},w}$]\label{Poin1}
Let $1\leq N_{\eta}\leq +\infty$. For every $u\in {\rm{Dom}}_{{\mathrm r},w}$,  there exists a family of points $x_{i}\in I_{a_{i},b_{i}}$ such that $\vert x_{i}\vert=\frac{a_i+b_i}{2}$, for $i=1,\ldots, N_{\eta}$ such that
\begin{equation}\label{poincareformula}
\sum_{i=1}^{N_{\eta}}\frac{\omega_{d}^{p-1}}{b_{i}-a_{i}}\int_{I_{a_{i},b_{i}}}
\left|u(\zeta)-u\left(x_{i}\right)\right|^{p}(\wstar_{p}(\zeta))^{p-1}\,\de \zeta\leq \int_{I_{\Omega,w}}|\nabla u(y)|^{p}w(y)\,\de y\,.
\end{equation}
\end{teo}
\begin{oss}\label{serra}
As in \cite[Theorem 4.11]{CC}, since $\tilde w_p\le\,\wstar_{p}$ on $\Omega$, inequality \eqref{poincareformula} holds with the weight $\tilde w_p$ in the left hand side, instead of $\wstar_{p}$. Since the results of this paper from this point on will be based on this Poincar\'e inequality,
we can assume that $\wstar_{p}$ is bounded (up to change it by $\tilde w_p$) and so by \eqref{incltildew} we can assume that
\begin{equation}\label{include}
L^p(\Omega)\subset L^{p}(\Omega,(\wstar_{p})^{p-1}).
\end{equation}
\end{oss}
\proof
The proof of this Proposition can be done by using the radiality of $u$, and the same reasoning used in the onedimensional case in \cite[Theorem 2.10]{CCMH1}. For the sake of completeness, we give the proof of \eqref{poincareformula}. Take $1\leq i \leq N_{\eta}$, and consider $x_{i}$ on the sphere of radius $\frac{a_i+b_i}{2}$. By in \eqref{a1}, one gets
\begin{equation*}
\left|u(\eta)-u\left(\frac{a_i+b_i}{2}\right)\right|^{p}(\wstar_{p}(\eta))^{p-1}\leq
\int_{a_i}^{\frac{a_i+b_i}{2}}|u'(y)|^{p}w(y)\,\de y.
\end{equation*}
\begin{align*}
|u(x_{i})-u(\zeta)|^{p}\omega_{d}^{p}\,(\wstar_{p}(\zeta))^{p-1}
\leq \int_{I_{\vert \zeta\vert,\frac{a_i+b_i}{2}}}|\nabla u(y)|^{p}w(y)\,\de y\,;
\end{align*}
Hence, by integrating with respect to $\zeta$ gives that
\begin{equation*}
\int_{I_{a_{i},\frac{a_i+b_i}{2}}}
\left|u(\zeta)-u\left(x_{i}\right)\right|^{p}\omega_{d}^{p}(\wstar_{p}(\zeta))^{p-1}\,\de \zeta\leq\frac{\omega_{d}(b_{i}-a_{i})}{2}
\int_{I_{a_{i},\frac{a_i+b_i}{2}}}|\nabla u(y)|^{p}w(y)\,\de y.
\end{equation*}
Further, by letting the same reasoning, we get
\begin{equation*}
\int_{I_{\frac{a_i+b_i}{2},b_{i}}}
\left|u(\zeta)-u\left(x_{i}\right)\right|^{p}\omega_{d}^{p}(\wstar_{p}(\zeta))^{p-1}\,\de \zeta\leq\frac{\omega_{d}(b_{i}-a_{i})}{2}
\int_{I_{\frac{a_i+b_i}{2},b_{i}}}|\nabla u(y)|^{p}w(y)\,\de y.
\end{equation*}
Therefore, we deduce that
\begin{equation*}
\int_{I_{a_{i},b_{i}}}
\left|u(\zeta)-u\left(x_{i}\right)\right|^{p}\omega_{d}^{p}(\wstar_{p}(\zeta))^{p-1}\,\de \zeta\leq\frac{\omega_{d}(b_{i}-a_{i})}{2}
\int_{I_{a_{i},b_{i}}}|\nabla u(y)|^{p}w(y)\,\de y
\end{equation*}
and our conclusion follows. Now, since $u\in {\rm{Dom}}_{{\mathrm r},w}$, then
\begin{equation*}
\sum_{i=1}^{N_{\eta}}
\int_{I_{a_{i},b_{i}}}|\nabla u(y)|^{p}w(y)\,\de y=\int_{I_{\Omega,w}}|\nabla u(y)|^{p}w(y)\,\de y<+\infty,
\end{equation*}
and we are done.
\qed

By following the same reasoning used in \cite{CCMH1}, we define 
\begin{align}
\label{eq:spaceW}
W=W(\Omega, w):=\,{\rm{Dom}}_{{\mathrm r},w}\cap L^{p}(\Omega,(\wstar_{p})^{p-1}).
\end{align}
In the next, we prove that $W$ endowed with a suitable norm is a Banach space.
\begin{prop} 
Let us consider $W$ be defined as in \eqref{eq:spaceW}, and endow it with the norm
\begin{align}\label{normapp}
\|u\|_W\coloneqq \sqrt[p]{\|u\|_{L^{p}(I_{\Omega,w},(\wstar_{p})^{p-1})}^{p}+\|\nabla u\|_{L^{p}(I_{\Omega,w},w)}^{p}}\quad\text{ if }u\in W\,.
\end{align}
Then $(W,\|u\|_W)$ is a Banach space. Further, if $w$ is a finitely degenerate weight, then
\begin{equation}\label{LipdenseW}
\text{ $AC_{r}^{d}(\Omega)$ is dense in $(W,\|\cdot\|_W)$} 
\end{equation}
in the following sense. For every $u\in W$ there exists a sequence $(u_h)_h\subset AC_{r}^{d}(\Omega)$ such that
\begin{equation*}\label{approxseqW}
\lim_{h\to\infty}u_h=\,u \text{ in } (W,\|\cdot\|_W)\,,
\end{equation*}
that is,
\begin{equation}\label{convinWsplitted}
\lim_{h\to\infty}u_h=\,u \text{ in } L^{p}(\Omega,(\wstar_{p})^{p-1}),\text{ and }\lim_{h\to+\infty}\nabla u_h=\,\nabla u \text{ in } L^{p}(I_{\Omega,w},w;\mathbb{R}^d)\,.
\end{equation}
\end{prop}
\begin{proof}
Its proof is a direct consequence of the radial condition and \cite[Proposition 2.11]{CCMH1}. Indeed, let us first observe that $W$ is a Banach space. 
Suppose that $(u_h)_h\subset (W, \|\cdot\|_W)$  is a Cauchy sequence. Hence by definition $u_{h}(x)=v_{h}(\vert x\vert)$, where $v_{h}\in W^{1,1}(\supp(\eta))\cap L^{p}(\supp(\eta),(\hat{\eta})^{p-1})$. Then by \cite[Proposition 2.11]{CCMH1}, the space 
\begin{align*}
\tilde{W}\coloneqq \,{\rm{Dom}}_\eta\cap L^{p}(\supp(\eta),(\hat{\eta}_{p})^{p-1})
\end{align*}
is a Banach space with norm 
\[
\|v\|_{\tilde{W}}\coloneqq \sqrt[p]{\|u\|_{L^{p}(I_{p,\supp(\eta)},(\hat{\eta}_{p})^{p-1})}^{p}+\|v'\|_{L^{p}(I_{p,\supp(\eta)},\eta)}^{p}}\quad\text{ if }v\in \tilde{W}\,.
\]
Hence, it follows that there exist $v\in L^{p}(I_{p,\supp(\eta)},(\hat{\eta}_{p})^{p-1})$, and $\tilde{v}\in L^{p}(I_{p,\supp(\eta)},\eta)$ such that,
\begin{equation}\label{CsW}
v_h\to v\text{ in  }L^{p}(I_{p,\supp(\eta)},(\hat{\eta}_{p})^{p-1}),\text{ and }v_{h}'\to \tilde{v}\text{ in  }L^{p}(I_{p,\supp(\eta)},\eta)\,,
\end{equation}
as $h\to+\infty$. Furthermore, for each $i=1,\dots,N_\eta$,
\begin{equation}\label{uinD*}
v\in AC((a_i,b_i)) \text{ and } v'=\,\tilde{v}\text{ a.e.  in } (a_i,b_i).
\end{equation}
Therefore, by the radiality of $(u_{h})$ we can write \eqref{CsW} in terms of $u_{h}$, and by \eqref{uinD*} we are done. 
\end{proof}
\section{Relaxation for finitely degenerate weights}\label{relaxp}
In this section (and in the following one) we will suppose that   $w$ is a finitely degenerate weight.
We consider $X=L^{p}(\Omega,(\wstar_{p})^{p-1})$ where $\wstar_{p}$ is the weight as defined in \eqref{pesop}. Let us set
\begin{equation*}
\!\!\!\!\!\! \!\!\! \!\!\!\!\!\! \! F(u)\coloneqq
\begin{cases}
\displaystyle\int_{\Omega}|\nabla u|^{p}\,w\,\de x &\text{ if } u\in AC_{r}^{d}(\Omega),\\
+\infty & \text{ if } u\in X\setminus  AC_{r}^{d}(\Omega)
\end{cases}
\end{equation*}
and thus we study the lower semicontinuous envelopes w.r.t.  $L^{p}((\wstar_{p})^{p-1})$-convergence, that is
$$
\overline{F}(u)={\rm sc^-}(L^{p}((\wstar_{p})^{p-1}))-F(u).
$$
In what follows, we let
$$
D\coloneqq\{u\in L^{p}(\Omega,(\wstar_{p})^{p-1}): \overline{F}(u)<+\infty
\}\,.
$$
We notice that, if $I_{\Omega, w}=\emptyset$, then $\wstar_{p}\equiv 0$ (see Proposition \ref{propw} (i)). Therefore, one gets that $L^{p}((a,b),(\wstar_{p})^{p-1})=\{0\}$, $D=\{0\}$ and $\overline{F}(u)=0$. In the next theorem, we then state an explicit formula for the relaxed functional $\overline{F}$ with respect to an opportune convergence.
\begin{teo} \label{mainfinitelocsum} 
We have
\begin{equation*}
D={\rm{Dom}}_{{\mathrm r},w}
\end{equation*}
where ${\rm{Dom}}_{{\mathrm r},w}$ is defined by \eqref{domw}
and the following representation holds for the relaxed functional
\begin{equation*}
\overline{F}(u)=
\begin{cases}
\displaystyle\int_{I_{\Omega, w}} |\nabla u|^{p}\,w\,\de x &\text{ if } u\in {\rm{Dom}}_{{\mathrm r},w}\cap L^{p}(\Omega,(\wstar_{p})^{p-1})\\
+\infty & \text{ if } u\in L^{p}(\Omega,(\wstar_{p})^{p-1})\setminus {\rm{Dom}}_{{\mathrm r},w} .
\end{cases}
\end{equation*}  
\end{teo}
\begin{proof} Note that  by  Lemma \ref{lemma2} and  Proposition \ref{poincarep},  we  deduce that  $D\subseteq {\rm{Dom}}_{{\mathrm r},w}$. Furthermore, for every $u\in D$ one gets
$$
u\in W^{1,1}_{\text {loc}}(I_{\Omega, w})\cap L^{p}(I_{\Omega, w}, (\wstar_{p})^{p-1}),\ \ u^{p}(\wstar_{p})^{p-1}\in L^\infty(I_{\Omega, w})\,.$$ 
In the next, we show that for every $u\in L^{p}(\Omega,(\wstar_{p})^{p-1})$
$$
\displaystyle\int_{I_{\Omega, w}} |\nabla u |^{p}\,w\,\de x\leq\overline{F}(u).
$$
By the definition of  $\overline{F}$, we directly suppose that $\overline{F}(u)<+\infty$. Therefore there exists a sequence $(u_h)\subset {\rm{Dom}}_{{\mathrm r},w}$ such that 
$u_h\to u$ in $L^{p}(\Omega,(\wstar_{p})^{p-1})$ and
$$
\overline{F}(u)=\lim_{h\to +\infty}F(u_h)=\lim_{h\to +\infty}\displaystyle\int_{\Omega} |\nabla u_h |^{p}\,w\,\de x.
$$
Then, thanks to Lemma \ref{lemma2} we get up to extracting a subsequence that
$$
\displaystyle\int_{\Omega} |\nabla u|^{p}\,w\,\de x\leq\liminf_{h\to +\infty}\displaystyle\int_{\Omega} |\nabla u_h |^{p}\,w\,\de x=\lim_{h\to +\infty}\displaystyle\int_{\Omega} |\nabla u_h |^{p}\,w\,\de x=\overline{F}(u)
$$
and we are done. To conclude, it remains to prove that
\begin{equation}\label{main1bb1}
\overline{F}(u)\le\,\displaystyle\int_{I_{\Omega, w}} |\nabla u|^{p}\,w\,\de x,\quad\forall u\in {\rm{Dom}}_{{\mathrm r},w}\,
\end{equation}
and thus ${\rm{Dom}}_{{\mathrm r},w}\subseteq D.$ In fact, this is a consequence of \eqref{LipdenseW}.  Indeed, by  property (i) in Corollary \ref{proppp1}  we have that that ${\rm{Dom}}_{{\mathrm r},w}\subset\,L^{p}(\Omega,(\wstar_{p})^{p-1})$. Thus,
if $u\in W={\rm{Dom}}_{{\mathrm r},w}\cap L^{p}(\Omega,(\hat{w}_{p})^{p-1})=\,{\rm{Dom}}_{{\mathrm r},w}$, by \eqref{LipdenseW}, there exists a sequence $(u_h)_h\subset AC_{r}^{d}(\Omega)$ such that
\eqref{convinWsplitted} holds true. Hence, from \eqref{convinWsplitted}, one has that
\begin{equation*}
\begin{split}
\overline{F}(u)&\le\,\liminf_{h\to\infty}{\overline{F}}(u_h)\le\,\lim_{h\to\infty}\int_{I_{\Omega, w}} |\nabla u_h |^{p}\,w\,\de x= \, \int_{I_{\Omega, w}} |\nabla u|^{p}\,w\,\de x\,,
\end{split}
\end{equation*}
and thus \eqref{main1bb1} holds true.
 \end{proof}
Let us define the following spaces:
\begin{align*}
C_{r}^{1}(\ov{\Omega})&\coloneqq \left\{u\in C^{1}(\ov{\Omega}): \text{u is radial in $\Omega$}\right\},
\\
{\rm Lip}_{r}(\ov{\Omega})&\coloneqq \left\{u\in {\rm Lip}(\ov{\Omega}): \text{u is radial in $\Omega$}\right\},
\\
{\mathcal{W}}^{1,p}_r(\Omega):&\coloneqq\left\{u\in AC_{r}^{d}(\Omega): \nabla u\in L^{p}(\Omega)\right\}.
\end{align*}
Notice that
\begin{align*}
{\mathcal{W}}^{1,p}_r(\Omega)\subset L^p_{\rm{loc}}(\Omega,(\wstar_{p})^{p-1}).
\end{align*} 
We consider the following functionals defined on the space $L^p(\Omega,(\wstar_{p})^{p-1})$
\begin{equation*}
\!\!\!\!\!\! \!\!\! \!\!\!\!\!\! \!\!\!   F^1(u):=
\begin{cases}
\displaystyle\int_{\Omega} |\nabla u|^p\,w\,\de x &\text{ if } u\in C_{r}^1(\ov{\Omega}),\\
+\infty & \text{ if } u\in L^p(\Omega,(\wstar_{p})^{p-1})\setminus C_{r}^{1}(\ov{\Omega}) ,
\end{cases}
\end{equation*}
\begin{equation*}
\!\!\!\!\!\! \!\!\! \!\!\!\!\!\! \!  F^2(u):=
\begin{cases}
\displaystyle\int_{\Omega}|\nabla u|^p\,w\,\de x &\text{ if } u\in {\rm Lip}_r(\ov{\Omega}),\\
+\infty & \text{ if } u\in L^p(\Omega,(\wstar_{p})^{p-1})\setminus {\rm Lip}_{r}(\ov{\Omega}),
\end{cases}
\end{equation*}
\begin{equation*}
\!\!\!\!\!\! \!\!\! \!\!\!\!\!\! \!\!  F^3(u)\coloneqq
\begin{cases}
\displaystyle\int_{\Omega} |\nabla u|^p\,w\,\de x &\text{if $u\in {\mathcal{W}}_{r}^{1,p}(\Omega)$}\\
+\infty & \text{ if $u\in L^p(\Omega,(\wstar_{p})^{p-1})\setminus {\mathcal{W}}_{r}^{1,p}(\Omega)$}.
\end{cases}
\end{equation*}
Since
$$
C^1_r(\ov{\Omega})\subset {\rm Lip}_r(\ov{\Omega})\subset {\mathcal{W}}_{r}^{1,p}(\Omega)\subset AC_{r}^{d}(\Omega),
$$
we have
$$
F^1(u)\leq F^2(u)\leq F^3(u)\leq F(u).
$$
Note also that, when $q=\infty$, $F^{2}$ and $F^{3}$ agree. Moreover, let us consider the corresponding lower semicontinuous envelopes w.r.t. the $L^p(\Omega,(\wstar_{p})^{p-1})$-convergence 
\begin{equation}
\label{newrelaxfunct}
\overline {F^j}(u)={\rm sc^-}(L^p(\Omega,(\wstar_{p})^{p-1}))-F_j(u)\ \ \ \ j=1,2,3,
\end{equation}
we have
$$
\overline{F^1}(u)\leq \overline{F^2}(u)\leq \overline{F^3}(u)\leq \overline{F}(u).
$$
\begin{cor}\label{equivalenze}
For every $u\in L^p(\Omega,(\wstar_{p})^{p-1})$ we have
$$
\overline{F^1}(u)=\overline{F^2}(u)=\overline{F^3}(u)=\overline{F}(u),
$$
where
$\overline{F^j}(u)$, $j=1,2,3$ are the functionals in \eqref{newrelaxfunct}.
\end{cor}
\proof 
This is direct consequence of \cite[Corollary 3.2]{CCMH1} and the radiality of the involved functions.
\qed
\begin{oss}
Let us point out that, as a consequence of Corollary ~\ref{equivalenze}, the closure of ${\rm Lip}_{r}(\ov{\Omega})$ with respect to the norm \eqref{normapp} is given by
\begin{align}\label{closurep}
\overline{{\rm Lip}_{r}(\ov{\Omega})}^{W}={\rm{Dom}}_{{\mathrm r},w}\cap L^{p}(\Omega,(\hat{w}_{p})^{p-1}).
\end{align}
On the other hand,  observe that for $p=2$, our Corollary \ref{equivalenze} recovers the one stated in \cite[Corollary 4.23]{CC} in the one-dimensional case, now extended to higher dimensions under radial symmetry assumptions. Furthermore, it is important to note that, according to \cite[Remark 3.4]{CC}, when $d\geq 2$, and $p=2$, there exists a non-radial weight for which the equivalences stated in Corollary ~\ref{equivalenze} no longer hold. In particular, the identity \eqref{closurep} may fail in the absence of radial symmetry. An explicit counterexample illustrating this phenomenon can be found in \cite[Example 2.2]{CPSC}.
\end{oss}
\begin{oss}
Let us observe that, when $w=1$, or more generally $w\geq C>0$, then $I_{\Omega,w}=\Omega$ and, since by Remark \ref{serra}
\begin{equation*}\label{include1}
L^p(\Omega)\subset L^{p}(\Omega,(\wstar_{p})^{p-1}),
\end{equation*}
we have
\begin{equation*}
\!\!\!\!\!\! \!\!\! \!\!\!\!\!\! \!\!\!   F(u):=
\begin{cases}
\displaystyle\int_{\Omega} |\nabla u|^p w\,\de x &\text{ if } u\in AC_{r}^{d}(\Omega),\\
+\infty & \text{ if } u\in L^p(\Omega,(\wstar_{p})^{p-1})\setminus AC_{r}^{d}(\Omega),
\end{cases}
\end{equation*}
\begin{equation*}
\!\!\!\!\!\!\!\!\!\!\!\!\!\!\!\!\!\!\!\!\!\!\!\!\! =
\begin{cases}
\displaystyle\int_{\Omega} |\nabla u|^p w\,\de x &\text{ if } u\in AC_{r}^{d}(\Omega),\\
+\infty & \text{ if } u\in L^p(\Omega)\setminus AC_{r}^{d}(\Omega).
\end{cases}
\end{equation*}
By Theorem \ref{mainfinitelocsum}  we have
\begin{equation*}
\overline{F}(u)=
\begin{cases}
\displaystyle\int_{\Omega} |\nabla u|^{p} w\,\de x &\text{ if } u\in {\rm{Dom}}_{{\mathrm r},w}\cap L^{p}(\Omega),\\
+\infty & \text{ if } u\in L^{p}(\Omega)\setminus {\rm{Dom}}_{{\mathrm r},w} .
\end{cases}
\end{equation*}  
By observing that ${\rm{Dom}}_{{\mathrm r},w}\cap L^{p}(\Omega)=W_{r}^{1,p}(\Omega)$ defined in \eqref{W1p},
we recover the classical results, i.e.
$$
\overline{F^1}(u)=\overline{F^2}(u)=\overline{F^3}(u)=\overline{F}(u)
=
\begin{cases}
\displaystyle\int_{\Omega} |\nabla u|^p w\,\de x &\text{ if } u\in W_{r}^{1,p}(\Omega),\\
+\infty & \text{ if } u\in L^p(\Omega)\setminus W_{r}^{1,p}(\Omega).
\end{cases}
$$
\end{oss}
\section{Existence of minimizers for radial degenerate variational problems}\label{sec:minimizer}
Let us consider the following integral functional
$$
H(u):=\overline{F}(u)+\|u-g\|_{L^{p}(\Omega,(\wstar_{p})^{p-1})}
$$
where $g$ is a radial function, $g\in L^{p}(\Omega,(\wstar_{p})^{p-1})$, defined in the Banach space
$
W$
(see \eqref{eq:spaceW}) equipped 
with the norm
$\|u\|_W$
as defined in \eqref{normapp}.
\begin{teo}\label{exist} There exists an unique minimizer $u_0$ for the minimum problem
$$
\min_W H(u),
$$
i.e. for every ``competitor" $z\in W$ we have
$$
H(u_0)\leq H(z).
$$
\end{teo}
\begin{proof} It is a consequence of the classical direct methods of the Calculus of Variations, since the functional $H(u)$ is coercive and lower semicontinuous with respect to the norm in $L^{p}(I_{\Omega,w},w)$. The uniqueness is due to the strict convexity of $H(u)$.
\end{proof}
Now, let us consider the larger space
\begin{align*}
{\rm{\overline{Dom}}}_{w}\coloneqq \left\{u: \Omega\rightarrow \R: u\in W_{{\rm loc}}^{1,1}(I_{\Omega,w}), \displaystyle\int_{I_{\Omega,w}}\vert \nabla u\vert^{p}w\de x<+\infty\right\}=\bigcap_{i=1}^{N_w}{\rm{\overline{Dom}^i}}_{w},
\end{align*}
where $I_{\Omega,w}$ is defined according to \eqref{Iomega}. We have
$$
{\rm{\overline{Dom}}}_{w}\coloneqq\bigcap_{i=1}^{N_w}{\rm{\overline{Dom}^i}}_{w},
$$
with
\begin{align*}
{\rm{\overline{Dom}^i}}_{w}\coloneqq \left\{u: \Omega\rightarrow \R: u\in W_{{\rm loc}}^{1,1}(I_{a_{i},b_{i}}),\displaystyle\int_{I_{a_{i},b_{i}}}\vert \nabla u\vert^{p}w\de x<+\infty\right\}.
\end{align*}
\begin{teo}\label{thm:minimizer}
Let $u_0$ as in Theorem \ref{exist}. Then function $u_0$ is the minimizer for the following minimum problem
$$
\min_{{\rm{\overline{Dom}}}_{w}} H(u),
$$
i.e. for every "competitor" $z\in {\rm{\overline{Dom}}}_{w}\cap L^{p}(\Omega,(\wstar_{p})^{p-1})$ we have
$$
H(u_0)\leq H(z).
$$
\end{teo}
\begin{proof} 
By assumption, for every $i=1,\dots,N_w$ 
$$
H(u^i_0)\leq H(z^i_{rad}),
$$
for every ``competitor" $$z^i_{rad}\in {\rm{{Dom}^i}}_{w}\cap L^{p}(I_{a_{i},b_{i}},(\wstar_{p})^{p-1}),$$
where $u^i_0$ is the restriction of $u_0$ to $I_{a_{i},b_{i}}$.
It is sufficient to prove  that, for every $i=1,\dots,N_w$  we have
$$
H(u^i_0)\leq H(z^i),
$$
for every ``competitor" $$z^i\in {\rm{\overline{Dom}^i}}_{w}\cap L^{p}(I_{a_{i},b_{i}},(\wstar_{p})^{p-1}).$$
On the other hand, fixed $i=1,\dots,N_w$, we can repeat the argument of the proofs of \cite{CrastaMalusa}[Theorem 3.2 and Corollary 3.3] with 
$(0,R)=(a_i,b_i)$ and $f(r,s)=\eta(r)|s|^p$. 
As proven in \cite{CrastaMalusa}[Corollary 3.3] the minimization problem for $H(u)$ on ${\rm{\overline{Dom}^i}}_{w}$
is equivalent to the minimization problem for the one–dimensional functional
$$F^i_{\rm{rad}}(u)
:=
\int_{a_i}^{b_i} r^{d-1} |u'(r)|^p\eta(r)\,dr,
$$
in the functional space
$$
{\mathcal{W}}^{i}_{\rm{rad}}:=\left\{u\in AC_{\rm{loc}}((a_i,b_i)):\ u(c_i)=d_i,\  r^{d-1}|u'(r)|^p\in L^1((a_i,b_i),\eta)
\right\},
$$
which plays the role of ${\mathcal{W}}^{1}_{\rm{rad}}$ of \cite{CrastaMalusa}.
By \cite{CrastaMalusa}[Theorem 3.2] for every competitor $z^i\in {\rm{\overline{Dom}^i}}_{w}\cap L^{p}(I_{a_{i},b_{i}},(\wstar_{p})^{p-1})$ (this space plays the role of $W_{{\rm loc}}^{1,1}(I_{a_i,b_i})$), there exists a radial function $z^i_{rad}\in {\rm{{Dom}^i}}_{rad,w}\cap L^{p}(I_{a_{i},b_{i}},(\wstar_{p})^{p-1})$ such that 
$$
H(z^i_{rad})\leq H(z^i).
$$
Then 
$$
H(u^i_0)\leq H(z^i_{rad})\leq H(z^i).
$$
\end{proof}
\begin{oss}
As noticed before in \eqref{ifp>d}, if $p>d$, then $W_{r}^{1,p}(\Omega) \subset {\mathrm AC}_{r}^{d}(\Omega)$. In this case, the space ${\mathrm AC}_{r}^{d}(\Omega)$ is larger than the space of $W_{r}^{1,p}(\Omega)$-functions. However, our minimum problem is set in the class of $W_{{\rm loc}}^{1,1}(\Omega)$-functions whose gradient belongs to $L_{{\rm loc}}^{p}(\Omega,w)$. Therefore, we cannot assume apriori that $u\in L_{{\rm loc}}^{p}(\Omega,w)$, and therefore the minimization cannot be directly carried out within $W_{r}^{1,p}(\Omega)$. Formulating the minimization in $W^{1,p}(\Omega)$ without the radiality constraint, would require an additional assumption on the integrability of $u$, namely $u\in L_{{\rm loc}}^{p}(\Omega,w)$, which does not hold in general in our setting. Accordingly, it is appropriate at this stage to set the minimization problem in the weighted Sobolev space $W^{1,p}(\Omega,\mu)$, although a rigorous treatment of this formulation is beyond of the present work and is left for future investigation.
\end{oss}
\begin{oss}
Let us give a comment about the weighted Sobolev space $W^{1,p}(\Omega,w)$  in order to point out that our approach is essentially different. As shown in \cite{LAMB0}, this space can be constructed for general metric measure spaces. That is, we may replace $(\mathbb{R}^{d}, \norma{\cdot})$ by a generic metric space $(X,{\rm d})$. The authors showed that at least three different approaches can be used: The ${\bf H}$-approach based on the density of Lipschitz functions, the ${\bf W}$-approach based on the integration by part formula, and the ${ \bf BL}$-approach based on the property of functions to be absolutely continuous along curves. In our approach, we have used the ${ {\bf W}}$-approach since we have considered as metric space $X=\mathbb{R}^{d}$, and the metric ${\rm d}$ induced by the usual norm. Let us notice that in the ${\bf H}$-approach the construction is as follows. Let $(X,{\rm d})$ be a metric mesure space where $\mu$ is a boundedly-finite measure.  In what follows, let us denote by ${\rm Lip}(f,X)$ 
\begin{align}\label{lipfnew}
{\rm Lip}(f,X)\coloneqq \sup\left\{\frac{\vert f(x)-f(y)\vert}{{\rm d}(x,y)}: x,y\in X,\, x\neq y\right\},
\end{align}
 We denote by ${\rm Lip}(X)$ the set of such functions for which \eqref{lipfnew} is finite.  Given $f\in {\rm Lip}(X)$ its upper gradient (or slope) is defined as 
\begin{align*}
\vert \nabla f\vert(x)\coloneqq \limsup_{y\rightarrow x}\frac{\vert f(y)-f(x)\vert}{{\rm d}(x,y)}.
\end{align*}
A non-negative function $g\in L^{p}(X,{\rm d},\mu)$ is said to be a relaxed $p$-upper slope of $f$ if there exist $(f_{n})_{n}\subseteq {\rm Lip}(X)$ of boundedly supported functions, and $0\leq g'\leq g$ such that
\begin{align*}
f_{n}\rightarrow f\, \text{in $L^{p}(X,{\rm d},\mu)$},\,\hskip 0,2cm {\rm Lip}_{a}(f_{n})\rightarrow g',\, \text{weakly in $L^{p}(X,{\rm d},\mu)$}
\end{align*}
where 
\begin{align*}
{\rm Lip}_{a}(f_{n})(x)\coloneqq
\begin{cases}
\inf_{r>0}{\rm Lip}(f_{n},B_{r}(x))& \text{if $x$ is an accumulation point,}\\
0& \text{if $x$ is an isolated point.}
\end{cases}
\end{align*}
Here $B_{r}(x)$ denotes a ball of radius $r$ centered at $x\in X$. By following \cite{LAMB0}, we denote by ${\rm RS}(f)$ the set of all possible $p$-upper slope, and 
\begin{align*}
\vert Df\vert\coloneqq \bigwedge \left\{g'\in L^{p}(X,{\rm d},\mu): g'\in {\rm RS}(f)\right\}
\end{align*}
the minimal relaxed $p$-upper slope. Then the weighted Sobolev space $W^{1,p}(X)$ is defined as
\begin{align*}
W^{1,p}(X)\coloneqq \left\{ f\in L^{p}(X,{\rm d},\mu): {\rm RS}(f)\neq \emptyset\right\}.
\end{align*}
Then the norm in  $W^{1,p}(X)$ is given by
\begin{align*}
\norma{f}_{W^{1,p}(X)}^{p}\coloneqq \norma{f}_{L^{p}(X,{\rm d},\mu)}^{p}+\norma{\vert Df\vert}_{L^{p}(X,{\rm d},\mu)}^{p}.
\end{align*}
We point out that in this approach, $\vert Df\vert$ is not longer a vector but a non-negative function. Moreover, by definition, the convergence of  ${\rm Lip}_{a}(f_{n})$ is considered with respect to the measure $\mu$. In contrast, our framework allows for convergence with respect to a different measure (say, $\hat{\mu}(\de x)\coloneqq \hat{w}_{p}(x)\de x$).
\end{oss}

\bigskip

\textsc{Acknowledgments.}
The authors are members of  the Istituto Nazionale di Alta Matematica (INdAM), GNAMPA Gruppo Nazionale per l'Analisi Matematica, la Probabilità e le loro Applicazioni, and are partially supported by the INdAM--GNAMPA 2023 Project \textit{Problemi variazionali degeneri e singolari} and the INdAM--GNAMPA 2024 Project \textit{Pairing e div-curl lemma: estensioni a campi debolmente derivabili e diﬀerenziazione non locale}.\\
This study was carried out within the ``2022SLTHCE - Geometric-Analytic Methods for PDEs and Applications (GAMPA)" project – funded by European Union – Next Generation EU  within the PRIN 2022 program (D.D. 104 - 02/02/2022 Ministero dell’Università e della Ricerca).\\
AMH has been supported by project PRIN 2022 
``understanding the LEarning process of QUantum Neural networks (LeQun)'', proposal code 2022WHZ5XH -- CUP J53D23003890006.
\\
This manuscript reflects only the authors’ views and opinions and the Ministry cannot be considered responsible for them.

\end{document}